%% file: grothendieck.tex
\documentclass[final,oneside,notitlepage,11pt]{article}

\input{styles}

\input{kobib}
\input{makros}

\input{hyphenation_en}
\input{theorems}

\title{Internal geometry and functors between sites}
\author{Konrad Waldorf}
\email{konrad.waldorf@uni-greifswald.de}

\adress{
  Universität Greifswald\\
  Institut für Mathematik und Informatik\\
  Walther-Rathenau-Str. 47\\
  D-17487 Greifswald}

\msc{}
\dedication{}
\pdfdaten

\denseversion
\allowdisplaybreaks

\def\mathscr#1{\EuScript{#1}}

\usepackage{amstext}
\usepackage{amssymb}
\usepackage{amsbsy}

\def\Uni{\mathcal{U}\hspace{-0.1em}n\hspace{-0.1em}i}
\def\Sing{\mathcal{S}\hspace{-0.1em}i\hspace{-0.1em}n\hspace{-0.1em}g}
\def\Pre{\mathcal{P}\hspace{-0.1em}r\hspace{-0.1em}e}

\def\Subman{\mathscr{S}\mathrm{ub}\Man}
\def\Cart{\mathscr{C}\mathrm{art}}
\def\Grpd{\mathscr{G}\mathrm{rpd}}
\def\Bun{\mathscr{B}\mathrm{un}}
\def\Bibun#1#2{#1\text{-}#2\text{-}\mathscr{B}\mathrm{i}\mathscr{B}\mathrm{un}}
\def\Fun{\mathscr{F}\mathrm{un}}
\def\Cat{\mathscr{C}\mathrm{at}}
\def\Ana{\mathscr{A}\mathrm{na}}

\def\Yo{\hspace{-0.25em}\raisebox{-0.1em}{\includegraphics[height=1em]{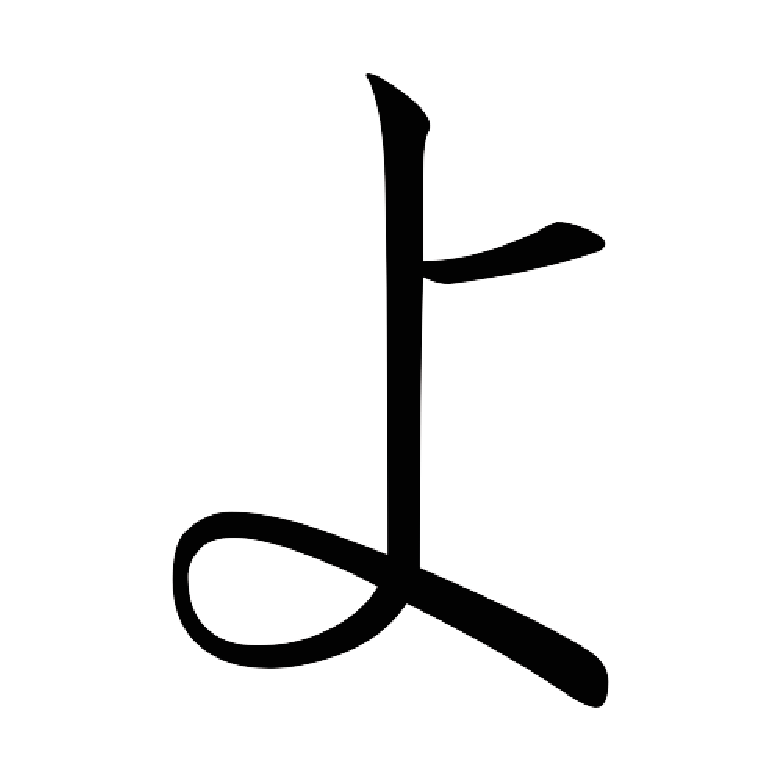}}\hspace{-0.15em}}

\def\nameEquation{\unskip}
\def\nameEquations{\unskip}

\usepackage{graphicx}

\begin{document}

\maketitle

\begin{abstract}
Locality is implemented in an arbitrary category using  Grothendieck topologies. We explore how different Grothendieck topologies on one category can be related, and, more general, how functors between categories can preserve them. As applications of locality, we review geometric objects such as sheaves,  groupoids, functors, bibundles, and anafunctors internal to an arbitrary Grothendieck site. We give definitions such that all these objects are invariant under equivalences of Grothendieck topologies and  certain functors between sites. As examples of sites, we look at categories of smooth manifolds, diffeological spaces, topological spaces, and sheaves, and we study properties of various functors between those.      
\end{abstract}


\setsecnumdepth{1}

\section{Introduction}

Grothendieck initiated a program in \cite{Grothendieck1960a} aimed at internalizing mathematical language, with the goal of transferring concepts from the category $\Set$ of sets to other categories $\mathscr{C}$. For instance, groups internal to the category $\Top$ of topological spaces are topological groups, while groups internal to the category $\Man$ of smooth manifolds are Lie groups, and so forth. This internalization approach offers two main advantages: first, it allows for the consistent combination of different mathematical structures, such as groups and smooth manifolds. Second, by keeping the category $\mathscr{C}$ arbitrary, it provides an efficient framework for defining structures and proving theorems that can be specialized to various contexts by choosing specific categories $\mathscr{C}$.

While much of the work on internalization has focused on algebraic structures, algebraic geometry, and logic, this article addresses the internalization of \emph{(differential-)geometric} structures, such as groupoids and principal bundles. These structures often involve a notion of locality. To implement locality within an arbitrary category $\mathscr{C}$, Grothendieck introduced the concept of a \quot{topology on a category}, now known as a Grothendieck topology. Many researchers have contributed to the internalization of geometric structures, with notable references including Grothendieck himself \cite{Grothendieck1960}, Ehresmann \cite{Ehresmann1963}, Bunge-Par\'e \cite{Bunge1979}, Bartels \cite{bartels}, Carchedi \cite{Carchedi2011}, and Roberts \cite{Roberts2012}. This program has proven extremely useful and insightful, particularly for categories with multiple Grothendieck topologies, such as the category $\Top$ of topological spaces and the category $\Diff$ of diffeological spaces.

Focusing specifically on internal groupoids and internal Morita equivalences, the work of Meyer and Zhu \cite{Meyer2014} provides a comprehensive  general framework. They describe two versions of internal Morita equivalences, one using bibundles and one using anafunctors, both of which give rise to equivalent bicategories, whose 1-isomorphisms are Morita equivalences. Moreover, the inclusion of the 2-category of internal functors maps weak equivalences (with respect to the Grothendieck topology) to Morita equivalences.
This framework has been described previously in each context separately, for example by Hilsum-Skandalis \cite{Hilsum1987} and Makkai \cite{makkai1} in topological spaces, by Blohmann in smooth manifolds \cite{Blohmann2008},  and van der Schaaf in diffeological spaces \cite{Schaaf2020}.

However, one aspect that is missing in the literature on the internalization of geometric structures is the comparison of different Grothendieck topologies on a category $\mathscr{C}$, and more generally, the ability to switch sites along functors. One reason for this oversight may be that existing terminology for comparing Grothendieck topologies is heavily rooted in algebraic geometry (where two Grothendieck topologies are considered equivalent if they yield the same categories of sheaves). Additionally, the notion of sieves, commonly used in this context, can appear unintuitive when applied to categories like $\Top$ or $\Man$. To address these issues, we propose and study in \cref{GTs} a new notion of equivalence between Grothendieck topologies (\cref{equivalence-between-grothendieck-topologies}): two Grothendieck topologies are considered equivalent if they determine the same class of \quot{universal locally split morphisms}. This concept is motivated by a discussion in \cite{Roberts2012} and aligns well with the categories of interest in this work:

\begin{itemize}
\item
Every Grothendieck topology on a category is equivalent to a singleton topology, called its universal completion (\cref{universal-completion}).

\item All standard Grothendieck topologies on the category $\Man$ of smooth manifolds are equivalent (\cref{topologies-on-man}), indicating that there is only one notion of locality on $\Man$.

\item
The category $\Top$ of topological spaces features at least four non-trivial equivalence classes of Grothendieck topologies (\cref{Comparison-of-GTs-on-Top}), each corresponding to a different notion of locality.

\item
Similarly, the category $\Diff$ of diffeological spaces contains at least four non-trivial equivalence classes of Grothendieck topologies (\cref{comparison-GTs-on-Diff}); to the best of our knowledge, this is the first systematic study of these equivalence classes.

\end{itemize}

\medskip

All forthcoming definitions and structures considered in this article will be invariant under equivalences of Grothendieck topologies. For example, we provide new definitions for when a functor between sites is \emph{continuous} (\cref{continuous-functor}), \emph{cocontinuous} (\cref{cocontinuous-functor}), and has a \emph{dense image} (\cref{dense-image}), all of which remain invariant under these equivalences.

In \cref{geometry-on-sites}, we revisit several internalized geometric structures. Our definitions differ from those typically found in the literature because existing definitions (1) are \emph{not} invariant under equivalences of Grothendieck topologies, and (2) often apply only to singleton Grothendieck topologies. We resolve these issues by applying the previous definitions to the universal completion of the given Grothendieck topology. For instance, rather than requiring that the projection map of a principal bundle be a covering, we require that it be universally locally split. While this may seem like a minor adjustment, it has three significant advantages:
\begin{enumerate}

\item
Our definitions are invariant under equivalences between Grothendieck topologies and are preserved by continuous functors (see \cref{continuous-functors-and-weak-equivalences,functoriality-of-G-bundles,continuous-functors-and-anafunctors}).

\item
The existing theory, such as the results of Meyer and Zhu \cite{Meyer2014}, can be used (see \cref{bicat-bibundles,bicategory-anafunctor,anafunctors-and-bibundles}).

\item
This approach avoids the artefact that arises when applying, for instance, the definition of a principal bundle from \cite{Meyer2014} to the Grothendieck topology of surjective local diffeomorphisms on the category $\Man$, in which case the bundle projection is required to be a local diffeomorphism. Our definition, by contrast, only requires it to \emph{split} over a surjective local diffeomorphism.

\end{enumerate}
\medskip

In \cref{sheaves}, we examine the extent to which the notion of a sheaf is compatible with our new equivalence relation between Grothendieck topologies. Unfortunately, it turns out that the traditional definition of a sheaf is not compatible. This result is unsurprising, given that the traditional definition does not even yield the same sheaves on $\Man$ for the Grothendieck topologies of open covers and surjective submersions (see \cref{traditional-sheaf-bad}), even though these are (and should be) equivalent. Therefore, with sincere apologies to the reader, we propose a modified definition of a sheaf (\cref{sheaf}). The key difference is that we add the condition that sheaves must send coproducts to products, a requirement that is often imposed in geometric contexts. This modification primarily affects singleton Grothendieck topologies; for superextensive (and \quot{singletonizable}) Grothendieck topologies, our new definition coincides with the traditional one (\cref{comparison-sheaf}). Our revised definitions of continuous and cocontinuous functors, as well as the revised  definition of a sheaf, necessitate then a reexamination of the classical base change operations for sheaves. By adding an additional condition on functors with respect to coproducts, we derive in \cref{adjunction-sheaves,comparison-lemma} versions of the \quot{comparison lemma}, similar to those found in \cite[App. 4]{MacLane1992} and \cite[C2.2]{Johnstone2002}.

\Cref{examples-of-sites} discusses the specific examples of sites mentioned earlier, and \cref{examples-of-functors} explores functors between these sites:

\begin{enumerate}[(a)]

\item
The inclusion of the site $\Open$ of open subsets of $\R^{n}$, $n\in \N$, into the site of smooth manifolds. This inclusion is continuous, cocontinuous, and has a dense image, thereby inducing equivalences between the categories of sheaves (\cref{sheaves-on-Open}).

\item
The smooth diffeology functor, which embeds smooth manifolds into diffeological spaces. This functor is continuous and cocontinuous for all four (equivalence classes of) Grothendieck topologies on diffeological spaces, but has a dense image only for the Grothendieck topology of subductions (\cref{properties-of-S}). This makes geometry on the sites $\Man$ and $\Diff$ highly compatible (\cref{smooth-diffeology-covariance,smooth-diffeology-sheaves}).

\item
The D-topology functor between diffeological spaces and topological spaces, which, however, turns out to be problematic in several respects.

\end{enumerate}

\medskip

This article originated from a talk I gave in 2022 at the Global Diffeology Seminar. While preparing for this talk, I realized that Grothendieck topologies had not yet gained traction within the diffeological community. I also noticed that important definitions, such as those of bicategories of groupoids, often reappear in different settings -- most recently in a diffeological context by van der Schaaf in \cite{Schaaf2020}. I hope that this article serves to summarize and integrate various strands of work within this field.

\paragraph{Acknowledgements.}
I would like to thank Severin Bunk and Alexander Engel for many help\-ful discussions on this topic.

\setsecnumdepth{2}

\section{Grothendieck topologies}

\label{GTs}

\subsection{Basic definitions}

In this article, all categories are locally small. 
If $\mathscr{C}$ is a category and $X\in\mathscr{C}$,
we denote by 
\begin{equation*}
\mathrm{Fam}_{\mathscr{C}}(X) := \{ (\pi_i)_{i\in I} \sep I\text{ a set, }U_i \in \mathscr{C}\text{, } \pi_i: U_i \to X \text{ a morphism}\}
\end{equation*}
the set of families of morphisms with target $X$. 

\begin{definition}
\label{grothendieck-topology}
A \emph{Grothendieck topology} on a category $\mathscr{C}$ is a family $T=(T_X)_{X\in \mathscr{C}}$, parameterized by the objects $X$ of $\mathscr{C}$, of subsets $T_X \subset \mathrm{Fam}_\mathscr{C}(X)$  whose elements are called the \emph{coverings} of $X$, 
such that:
\begin{itemize}

\item 
 Isomorphisms are coverings: if $\phi: U \to X$ is an isomorphism, then   $(\phi)\in T_X$.

\item
 Coverings of coverings are coverings: if $(\pi_i: U_i \to X)_{i\in I} \in T_X$ and for each $i\in I$, $(\psi_{j})_{j\in J_i}\in T_{U_i}$, then $(\pi_i \circ \psi_j)_{i\in I,j\in J_i} \in T_X$.   

\item
Coverings pull back to coverings: if $(\pi_i: U_i \to X)_{i\in I} \in T_X$  and $f: W \to X$ is a morphism, then the pullback
\begin{equation*}
\alxydim{@=2.5em}{f^{*}U_i \ar[r] \ar[d]_{f^{*}\pi_i} & U_i \ar[d]^{\pi_i} \\ W \ar[r]_{f} &X}
\end{equation*}
exists for each $i\in I$, and $(f^{*}\pi_i)_{i\in I} \in T_W$. 

\end{itemize} 
A category $\mathscr{C}$ together with a Grothendieck topology is called a \emph{site}. A covering is called \emph{singleton covering}, if its index set is a singleton.
A Grothendieck topology is called \emph{singleton Grothendieck topology}, if all coverings are singleton coverings.  
\end{definition}

What we defined above is often called a \emph{pretopology}. We only use pretopologies here, so that there is no need to distinguish. The following class of morphisms of a category will be important. 

\begin{definition}
A morphism in a category $\mathscr{C}$ is called \emph{universal} if its pullback along an arbitrary morphism exists. 
\end{definition}

\begin{remark}
\begin{enumerate}[(a)]

\item
If $\mathscr{C}$ admits has finite fibre products, every morphism in universal. 

\item 
All members of coverings of a Grothendieck topology, in particular: all isomorphisms, are universal. 

\item
The class of universal morphisms is stable under pullback and closed under composition  (this follows from the pasting law for pullbacks).

\end{enumerate}
\end{remark}
 
 \begin{example}
All universal morphisms of a category $\mathscr{C}$ form a singleton Grothendieck topology on $\mathscr{C}$, called the \emph{discrete Grothendieck topology} and denoted by $T_{dis}$. 
The \emph{indiscrete Grothendieck} topology $T_{indis}$ on a category $\mathscr{C}$ is the singleton Grothendieck topology  that  consists of all isomorphisms. 
If $T$ is an arbitrary singleton Grothendieck topology on $\mathscr{C}$, then 
\begin{equation*}
 T_{indis}\subset T \subset T_{dis}\text{.}
\end{equation*} 
\end{example}

\begin{remark}
Above, by $T \subset T'$ we mean that for each object $X\in \mathscr{C}$ we have $T_X \subset T_X'$, i.e., every covering of $T$ is also a covering of $T'$. Often, this way of comparing Grothendieck topologies is too strict. In \cref{equivalence-between-GTs} we introduce another relation that is weaker.  
\end{remark}

\begin{example}
\label{restriction}
If $(\mathscr{C},T)$ is a site and $\mathscr{D} \subset \mathscr{C}$ is a category such that the inclusion preserves finite fibre products (e.g., a reflexive subcategory), then we consider coverings  $(\pi_i:U_i \to X)_{i\in I}$ such that each  $\pi_i$ is universal in $\mathscr{D}$. These form a Grothendieck topology on $\mathscr{D}$, which we call the \emph{restriction of $T$ to $\mathscr{D}$} and denote it by $T|_{\mathscr{D}}$. 
\end{example}

\begin{definition}
In any category, a morphism $\pi:Y \to M$ is called an \emph{effective epimorphism}, if the double fibre product $Y^{[2]}=Y \ttimes\pi\pi Y$ exists and  the diagram  
\begin{equation*}
\alxydim{}{Y^{[2]} \arr[r] & Y \ar[r] & X}
\end{equation*}
is a coequalizer. A morphism is called a \emph{universally effective epimorphism}, if it is an effective epimorphism, universal, and every pullback is again a universally effective epimorphism.  
\end{definition}

\begin{example}
If $\mathscr{C}$ is any category, the class of universally effective epimorphisms forms a Grothendieck topology called the \emph{canonical Grothendieck topology} on $\mathscr{C}$ and is denoted by $T_{\mathscr{C}}$. 
\end{example}

\begin{definition}
\label{subcanonical}
A Grothendieck topology $T$ on $\mathscr{C}$ is called \emph{subcanonical} if $T \prec T_{\mathscr{C}}$. 
\end{definition}

This way of defining subcanonical Grothendieck topologies is evidently invariant under equivalences of Grothendieck topologies.

\begin{remark}
\label{epimorphisms}
Coequalizers, and in particular, effective epimorphisms, are epimorphisms in the categorical sense.
When $T$ is a subcanonical  Grothendieck topology, then all universal $T$-locally split morphisms are epimorphisms. When $T$ is subcanonical and  singleton, then all $T$-coverings are epimorphisms, matching the intuition that coverings are surjective. 
\end{remark}

\subsection{Equivalence}

\label{equivalence-between-GTs}

The following notion appears in
\cite[III 7.5]{MacLane1992} and \cite[Def. 3.8]{Roberts2012}, though under different names with another intention. 

\begin{definition}
Let $(\mathscr{C},T)$ be a site.
A morphism $\pi:Y \to X$ in $\mathscr{C}$ is called  \emph{locally split} if  there exists a covering $(\pi_i: U_i \to X)_{i\in I}$ together with  morphisms $\rho_i: U_i \to Y$ such that the diagram
\begin{equation*}
\alxydim{}{& Y \ar[d]^{\pi} \\ U_i \ar@/^1pc/[ur]^{\rho_i} \ar[r]_{\pi_i} & X}
\end{equation*}
is commutative for all $i\in I$. The morphisms $\rho_i$ are called  \emph{local sections}.
\end{definition}

\begin{example}
The locally split morphisms for the indiscrete Grothendieck topology $T_{indis}$ are precisely the \emph{globally split} morphisms (i.e., those morphisms $\pi:Y \to X$ for which there exists a morphism $\rho:X \to Y$ such that $\pi \circ \rho = \id_X$). 
\end{example}

\begin{lemma}
\label{locally-split-morphisms}
For every site  $(\mathscr{C},T)$ the following basic statements hold:
\begin{enumerate}[(a)]

\item
\label{locally-split-morphisms:a}
Every singleton covering is locally split.

\item 
\label{locally-split-morphisms:b}
The class of locally split morphisms is closed under composition.

\item
\label{locally-split-morphisms:c}
If $f: X \to Y$ and $g:Y \to Z$ are morphisms such that $g \circ f$ is a locally split, then $g$ is locally split.

\item
\label{locally-split-morphisms:d}
The class of locally split morphisms is stable under pullback, whenever these exist.

\item
\label{locally-split-morphisms:e}
Every globally split morphism  is locally split.

\end{enumerate}
\end{lemma}

In the following, when we consider more than one Grothendieck topology on one category $\mathscr{C}$, we use the terminology \quot{$T$-covering} and \quot{$T$-locally split}, in order to refer to a specific Grothendieck topology. 

\begin{definition}
\label{equivalence-between-grothendieck-topologies}
Let $\mathscr{C}$ be a category and let $T_1$ and $T_2$ be Grothendieck topologies on $\mathscr{C}$. We say that $T_1$ is  \emph{coarser} than $T_2$ if every universal $T_1$-locally split morphism is $T_2$-locally split. We write $T_1 \prec T_2$. Equivalently, we say that $T_2$ is \emph{finer} than $T_1$. Finally, we say that $T_1$ and $T_2$ are \emph{equivalent} if $T_1\prec T_2$ and $T_2\prec T_1$, and we denote this by $T_1 \sim T_2$. 
\end{definition}

\begin{remark}
If $T_1 \subset T_2$, i.e., every $T_1$-covering  is a $T_2$-covering, then $T_1\prec T_2$.
\end{remark}

\begin{remark}
\label{equivalence-by-refinement}
If $T_1$ and $T_2$ are Grothendieck topologies, then a \emph{$T_2$-refinement} of a $T_1$-covering $(\pi_i: U_i \to X)_{i\in I}$ is a $T_2$-covering $(\psi_j: V_j \to X)_{j\in J}$ together with a map $f: J \to I$ and morphisms $\rho_j: V_j \to U_{f(j)}$ such that $\pi_{f(j)} \circ \rho_j = \psi_j$. If every $T_1$-covering has a $T_2$-refinement, then $T_1 \prec T_2$. In particular, if $T_2 \subset T_1$ and every $T_1$-covering has a $T_2$-refinement, then $T_1 \sim T_2$.
\end{remark}

\begin{lemma}
\label{checking-equivalence}
Suppose $T_1$ is a singleton Grothendieck topology. Then, $T_1 \prec T_2$ if and only if every $T_1$-covering is $T_2$-locally split.  
\end{lemma}

\begin{proof}
Suppose every $T_1$-covering is $T_2$-locally split. If $\pi:Y \to X$ is $T_1$-locally split, let $\pi':Y' \to X$ be a $T_1$-covering with a section $\rho:Y' \to Y$. By assumption $\pi'$ is $T_2$-locally split. Thus, there exists a $T_2$-covering  $(\pi_i: U_i \to X )_{i\in I}$ with local sections $\rho_i: U_i\to Y'$. The composition $\rho \circ \rho_i$ exhibits $\pi$ as $T_2$-locally split. 
Conversely, assume that every $T_1$-locally split morphism is $T_2$-locally split. By \cref{locally-split-morphisms:a}, every $T_1$-covering is then $T_2$-locally split. 
\end{proof}

\begin{remark}
Every Grothendieck topology $T$ satisfies $T_{indis} \prec T \prec T_{dis}$.
\end{remark}

If $T$ is any Grothendieck topology, then  the class  $\Uni(T)$ of all universal $T$-locally split morphisms is  a singleton Grothendieck topology,  introduced by Roberts in  \cite[Def. 3.8]{Roberts2012}. It is called the \emph{universal completion of $T$}.

\begin{lemma}
\label{universal-completion}
The universal completion of a Grothendieck topology $T$ has the following properties:
\begin{enumerate}[(a)]

\item
\label{universal-completion:a}
$\Uni(T)\sim T$. 
In particular, every Grothendieck topology is equivalent to a singleton one.

\item
\label{universal-completion:c}
If $T_1$ and $T_2$ are Grothendieck topologies, then
\begin{equation*}
T_1 \prec T_2\quad\gdw\quad \Uni(T_1) \subset \Uni(T_2)
\end{equation*}
and
\begin{equation*}  
  T_1\sim T_2 \quad\gdw\quad  \Uni(T_1) = \Uni(T_2)\text{.}
\end{equation*} 

\item
\label{universal-completion:d}
If $T$ is singleton, then $T \subset \Uni(T)$.

\item 
\label{universal-completion:b}
$\Uni(\Uni(T))=\Uni(T)$.

\end{enumerate}
\end{lemma}

\begin{definition}
\label{local-GT}
A Grothendieck topology $T$ is called \emph{local}, if the following holds: if
\begin{equation*}
\alxydim{}{V \ar[r] \ar[d]_{f} & Y \ar[d]^{\pi} \\ U \ar[r]_{g} & X}
\end{equation*}
is a pullback diagram, and $f$ and $g$ are universal  $T$-locally split, then $\pi$ is universal  $T$-locally split. 
\end{definition}

\begin{remark}
\label{assumption-2.6}
Due to \cref{universal-completion:c}, being local is invariant under equivalences of Grothendieck topologies. In the terminology of \cite{Meyer2014}, $T$ is local in the sense of \cref{local-GT} if the singleton Grothendieck topology  $\Uni(T)$ satisfies \quot{Assumption 2.6}. 
\end{remark}

\subsection{Singletonization}

\label{singletonization-sec}

We discuss another, more common, method to produce a singleton Grothendieck topology $\Sing(T)$ from a non-singleton Grothendieck topology $T$, such that $T \sim \Sing(T)$. Typically, $\Sing(T)$ will be much smaller than $\Uni(T)$; on the other hand, its existence requires some conditions. 

\begin{definition}
\label{extensive-category}
A category $\mathscr{C}$ is called \emph{extensive}, if it satisfies the following conditions:
\begin{enumerate}[(a)]
\item 
It has an initial object $\emptyset$, 

\item 
Finite coproducts are disjoint, i.e., if $X,Y\in \mathscr{C}$ have a coproduct, then 
 \begin{equation*}
\alxydim{}{\emptyset \ar[r] \ar[d] & X \ar[d] \\ Y \ar[r] & X \coprod Y}
\end{equation*} 
is a pullback diagram.

\item
Coproducts are stable under pullback, i.e., if $(U_i)_{i\in I}$ is a family of objects whose coproduct $\coprod_{i\in I} U_i$ exists, and $f:X \to \coprod_{i\in I} U_i$ is a morphism, then the pullbacks
\begin{equation*}
\alxydim{}{X \ttimes f{\iota_i}U_i \ar[r]\ar[d] & U_i \ar[d]\\ X \ar[r]_-{f} & \displaystyle\coprod_{i\in I} U_i} 
\end{equation*}
exist for all $i\in I$, and
\begin{equation*}
X = \coprod_{i\in I} X \ttimes f{\iota_i} U_i
\end{equation*}
in virtue of  coproduct injections $X \ttimes f{\iota_i} U_i \to X$.
\end{enumerate}
\end{definition}

Note that \cref{extensive-category} is a bit weaker then the usual version: we do not not require that $\mathscr{C}$ has (finite) coproducts, we only require an initial object.

\begin{definition}
A Grothendieck topology $T$ on a category $\mathscr{C}$ is called \emph{singletonizable}, if for every $T$-covering $(\phi_i: U_i \to X)_{i\in I}$ 
 the coproduct of the family $(U_i)_{i\in I}$ exists.   
\end{definition}

If $T$ is singletonizable, we construct a singleton Grothendieck topology  $\Sing(T)$ whose coverings are the morphisms
\begin{equation*}
\phi:\coprod_{i\in I} U_i \to X
\quomma
\phi|_{U_i} := \phi_i
\end{equation*}
produced from all covering families $(\phi_i: U_i \to X)_{i\in I}$ of $T$. If $\mathscr{C}$ is extensive, this indeed defines a Grothendieck topology.
The main point of $\Sing(T)$ is the following fact, which follows directly from the construction. 

\begin{proposition}
\label{singletonization}
Let $\mathscr{C}$ be an extensive category and let $T$ be a singletonizable Grothendieck topology. Then, $\Sing(T)\sim T$. 
\end{proposition}

\begin{remark}
We have $\Sing(T) \subset \Uni(T)$, and typically $\Sing(T)$ is much smaller then $\Uni(T)$.
\end{remark}

\begin{remark}
\label{extensive-topology}
If the category $\mathscr{C}$ is extensive, then $\mathscr{C}$ has a canonical Grothendieck topology $T^{ext}_{\mathscr{C}}$ called the \emph{extensive topology}, consisting of all families $(X_i \to \coprod_{i\in I}X_i)_{i\in I}$ of coproduct injections, for all families $(X_i)_{i\in I}$ of objects $X_i \in \mathscr{C}$ whose coproduct exists. A Grothendieck topology $T$ is called \emph{superextensive} if $T^{ext}_{\mathscr{C}} \subset T$. We remark:
\begin{itemize}

\item
The empty family is a covering in $T^{ext}_{\mathscr{C}}$ of $\emptyset\in \mathscr{C}$. 

\item 
Singleton Grothendieck topologies are typically never superextensive. 

\item
$T^{ext}_{\mathscr{C}}$ is singletonizable, and $\Sing(T_{\mathscr{C}}^{ext})=T^{indis}$. Hence, $T^{ext}_{\mathscr{C}} \sim T^{indis}$.

\item
Every Grothendieck topology $T$ on an extensive category satisfies $T^{ext}_{\mathscr{C}} \prec T$. 
\end{itemize}
As the latter two points suggest, our equivalence relation \cref{equivalence-between-grothendieck-topologies} is two coarse to distinguish superextensiveness.
\end{remark}

\setsecnumdepth{2}

\subsection{Continuous functors}

\label{continuous-functors}

\begin{definition}
\label{continuous-functor}
If $(\mathscr{C}_1,T_1)$ and $(\mathscr{C}_2,T_2)$ are sites, then a functor $\mathscr{F}: \mathscr{C}_1\to \mathscr{C}_2$ is called \emph{continuous}, if
\begin{enumerate}[(i)]

\item 
it sends universal $T_1$-locally split morphisms to universal $T_2$-locally split morphisms.

\item
it preserves fibre products with universal $T_1$-locally split morphisms. 

\end{enumerate}
\end{definition}

\begin{remark}
A functor is continuous in the sense of \cref{continuous-functor} if and only if it sends $\Uni(T_1)$-coverings to $\Uni(T_2)$-coverings, and preserves fibre products of $\Uni(T_1)$-coverings. It is hence continuous (w.r.t. $T_1$ and $T_2$) if and only if it is continuous in the traditional sense (see, e.g.,  \cite[Def. 7.13.1]{stacks-project}) w.r.t. $\Uni(T_1)$ and $\Uni(T_2)$.
\end{remark}

Whether or not a functor is continuous depends only on the equivalence classes of the involved Grothendieck topologies -- this is the main reason for not using the traditional definition. The following lemma provides a useful criterion to check continuity.

\begin{lemma}
\label{continuous-functor-coverings}
If $\mathscr{F}$ sends $T_1$-coverings to $T_2$-coverings, preserves universal morphisms and  fibre products  with universal morphisms, then it is continuous. 
\end{lemma}

\begin{proof}
If $\pi:Y \to X$ is $T_1$-locally split, let $(\pi_i:U_i \to X)_{i\in I}$ be a $T_1$-covering with local sections $\rho_i: U_i \to Y$. By assumption, $(\mathscr{F}(\pi_i):\mathscr{F}(U_i) \to \mathscr{F}(X))_{i\in I}$ is a $T_2$-covering, and we have the local sections $\mathscr{F}(\rho_i)$ showing that $\mathscr{F}(\pi)$ is $T_2$-locally split. Since $\mathscr{F}$ preserves universal morphisms, $\mathscr{F}(\pi)$ is universal. This shows (i). (ii) follows directly from the assumptions.
\end{proof}

\begin{lemma}
\begin{enumerate}[(a)]
\item 
\label{identity-functor-continuous}
If $T_1$ and $T_2$ are Grothendieck topologies on a category $\mathscr{C}$, then  $T_1 \prec T_2$ if and only if the identity functor $\id: (\mathscr{C},T_1) \to (\mathscr{C},T_2)$ is continuous.

\item
\label{composition-of-continuous-functors}
The composition of continuous functors is continuous.

\end{enumerate}
\end{lemma}

\begin{remark}
\label{induced-topology-on-subcategory}
If $\mathscr{B}\subset \mathscr{C}$ is a  subcategory of a site $(\mathscr{C},T)$, such that the inclusion functor preserves universal morphisms and  fibre products with universal morphisms, then the inclusion functor is continuous for the restriction $T|_{\mathscr{B}}$ of $T$ to $\mathscr{B}$ (see \cref{restriction}).
\end{remark}

\subsection{Cocontinuous functors}

\label{cocontinuous-functors}

\begin{definition}
\label{cocontinuous-functor}
If $(\mathscr{C}_1,T_1)$ and $(\mathscr{C}_2,T_2)$ are sites, then a functor $\mathscr{F}: \mathscr{C}_1\to \mathscr{C}_2$ is called \emph{cocontinuous}, if for every object $X$ of $\mathscr{C}_1$ and every universal $T_2$-locally split morphism $\pi:Y \to \mathscr{F}(X)$ there exists a universal $T_1$-locally split morphism $\pi':Y' \to X$ and a refinement 
\begin{equation*}
\alxydim{}{\mathscr{F}(Y') \ar[dr]_{\mathscr{F}(\pi')} \ar[rr] && Y \ar[dl]^{\pi} \\ & \mathscr{F}(X)}
\end{equation*}
\end{definition}

\begin{remark}
A functor is cocontinuous w.r.t. $T_1$ and $T_2$ if and only if it is cocontinuous in the traditional sense (see, e.g. \cite[Def. 7.20.1]{stacks-project}) w.r.t. $\Uni(T_1)$ and $\Uni(T_2)$.
\end{remark}

\begin{lemma}
\label{check-cocontinuity}
Suppose $\mathscr{F}$ has the property that for every $X\in \mathscr{C}_1$ and every $T_2$-covering $(\pi_i: U_i \to \mathscr{F}(X))_{i\in I}$ there exists a $T_1$-covering $(\psi_j : V_j \to X)_{j\in J}$, a map $r: J \to I$, and morphisms $\rho_j: V_j \to U_{r(j)}$ such that $\pi_{r(j)} \circ \rho_j = \psi_j$. Suppose further that either
\begin{enumerate}[(a)]

\item
$T_1$ and $T_2$ are singleton, or 

\item
$T_1$ is singletonizable and $\mathscr{F}$ preserves coproducts. 

\end{enumerate} 
Then, $\mathscr{F}$ is cocontinuous.
\end{lemma}

\begin{proof}
Suppose $\pi:Y \to \mathscr{F}(X)$ is a universal $T_2$-locally split map. In case (a), it can be refined through a $T_2$-covering $\zeta: Z \to \mathscr{F}(X)$ via a morphism $Z \to Y$. By assumption, this $T_2$-covering $\zeta$ can be refined by the image of a $T_1$-covering $\pi':Y' \to X$ via a morphism $\mathscr{F}(Y') \to Z$. Since $\pi'$ is  universal $T_1$-locally split, and we have $\mathscr{F}(Y') \to Z \to Y$, this shows cocontinuity.

In case (b), $\pi$ can be refined through a $T_2$-covering $\zeta_i: Z_i \to \mathscr{F}(X)$ via morphisms $\rho_i:Z_i \to Y$. By assumption, this $T_2$-covering can be refined by the image of a $T_1$-covering $\pi_j':Y_j' \to X$ via a map $r:J \to I$ and morphisms $\kappa_j: \mathscr{F}(Y'_j) \to Z_{r(j)}$. 
Since $T_1$ is singletonizable, we may form the $T_1$-locally split morphism $Y':=\coprod Y_j' \to X$, and since $\mathscr{F}$ preserves coproducts, we obtain a morphism
\begin{equation*}
\mathscr{F}(Y') \cong \coprod \mathscr{F}(Y_j') \to \coprod Z_{r(j)} \to Y\text{,}
\end{equation*}
showing cocontinuity.  
\end{proof}

\begin{remark}
The composition of cocontinuous functors is cocontinuous.
\end{remark}

\begin{lemma}
\label{identitiy-functor-is-continuous}
If $T_1$ and $T_2$ are Grothendieck topologies on a category $\mathscr{C}$, then $T_2 \prec T_1$ if and only if the identity functor $\id: (\mathscr{C},T_1) \to (\mathscr{C},T_2)$ is cocontinuous.
\end{lemma}

\begin{proof}
If $T_2 \prec T_1$, i.e., every universal $T_2$-locally split morphism is universal $T_1$-locally split, providing the required condition. Conversely, if $\id$ is cocontinuous, and $\pi:Y \to X$ is universal  $T_2$-locally split, then there exists a universal $T_1$-locally split map $\pi':Y' \to X$ and a refinement map $Y' \to Y$, showing that $\pi$ is  $T_1$-locally split. 
\end{proof}

\subsection{Denseness}

\begin{definition}
\label{dense-image}
Let $\mathscr{C}_1$ and $\mathscr{C}_2$ be categories, and let $T$ be a Grothendieck topology on $\mathscr{C}_2$. 
A functor $\mathscr{F}: \mathscr{C}_1 \to \mathscr{C}_2$ is said to have a \emph{dense image} if:
\begin{enumerate}[(a)]

\item 
\label{dense-image:1}
For each object $X$ in $\mathscr{C}_2$ there exists a family $(U_i)_{i\in I}$ of objects $U_i \in \mathscr{C}_1$ and a family of morphisms $\phi_i: \mathscr{F}(U_i) \to X$ such that the coproduct $\coprod \mathscr{F}(U_i)$ exists and 
\begin{equation*}
\coprod_{i\in I} \mathscr{F}(U_i) \to X
\end{equation*} 
is universal  $T$-locally split.

\item
\label{dense-image:2}
It is full and faithful. 

\end{enumerate}
\end{definition}

\begin{remark}
If $\mathscr{F}$ is an equivalence of categories, then it has dense image not matter what $T$ is. 
\end{remark}

\begin{remark}
The standard definition of dense image would be to require that $(\phi_i)_{i\in I}$ is a covering of $X$. This is unsuitable for singleton Grothendieck topologies, and also not invariant under equivalences.
However, if $(\phi_i)_{i\in I}$ is a $T$-covering and $T$ is singletonizable, then $\mathscr{F}$ has dense image. 
\end{remark}

\begin{lemma}
Suppose $\mathscr{C}_1$, $\mathscr{C}_2$, and $\mathscr{C}_3$ are categories, $T_2$ is a Grothendieck topology on $\mathscr{C}_2$ and  $T_3$ is a  Grothendieck topology on $\mathscr{C}_3$, with the property that its universal $T$-locally split morphisms are closed under coproducts. If
 $\mathscr{F}:\mathscr{C}_1 \to \mathscr{C}_2$ is a functor with dense image, and $\mathscr{G}:\mathscr{C}_2 \to \mathscr{C}_3$ is a coproduct-preserving, continuous functor with dense image, then, the composition $\mathscr{G} \circ \mathscr{F}$ has dense image.
\end{lemma}

\begin{proof}
Condition \cref{dense-image:1*} is clear, we focus on \cref{dense-image:2*}.
Let $X\in \mathscr{C}_3$.
Since $\mathscr{G}$ has dense image, there exists a family $(U_i)_{i\in I}$ of objects $U_i \in \mathscr{C}_2$ and a family of morphisms $\phi_i: \mathscr{G}(U_i) \to X$ such that the coproduct $\coprod \mathscr{G}(U_i)$ exists and 
\begin{equation*}
\phi:\coprod_{i\in I} \mathscr{G}(U_i) \to X
\end{equation*} 
is universal  $T_3$-locally split.
Since $\mathscr{F}$ has dense image, there exists, for each $i\in I$, a family $(V_j)_{j\in J_i}$ of objects $V_j \in \mathscr{C}_1$ and a family of morphisms $\psi_j: \mathscr{F}(V_j)\to U_i$ such that the coproduct $\coprod \mathscr{F}(V_j)$ exists and 
\begin{equation*}
\psi_i:\coprod_{j\in J_i} \mathscr{F}(V_i) \to U_i
\end{equation*} 
is universal  $T_2$-locally split.
We consider the morphism\begin{equation*}
\alxydim{@C=4em}{\coprod_{i\in I}\coprod_{j\in J_i} \mathscr{G}(\mathscr{F}(V_j)) \ar[r]^{\cong} & \coprod_{i\in I} \mathscr{G}(\coprod_{j\in J_i} \mathscr{F}(V_j))\ar[r]^-{\coprod \mathscr{G}(\psi_i)} & \coprod_{i\in I} \mathscr{G}(U_i)\ar[r]^-{\phi} & X\text{.}}
\end{equation*}
Here, the first morphism is induced from the coproduct injections $\mathscr{F}(V_j) \to \coprod \mathscr{F}(V_j)$, and it is an isomorphism, hence universal $T_3$-locally split, because $\mathscr{G}$ is coproduct-preserving. The second morphism is a coproduct of universal  $T_3$-locally split morphisms, because $\mathscr{G}$ is continuous and $\psi_i$ is universal $T_2$-locally split, and hence again universal $T$-locally split by assumption. Since universal morphisms and locally split morphisms are closed under composition, the above morphism is universal $T_3$-locally split. 
\end{proof}

\setsecnumdepth{2}

\section{Geometry on sites}

\label{geometry-on-sites}

We want to describe some common geometric notions internal to a category $\mathscr{C}$ with a Grothendieck topology. We mainly rely on existing definitions and results, mostly by Bartels  \cite{bartels}, Roberts  \cite{Roberts2012}, and Meyer-Zhu \cite{Meyer2014}. Since, in none of these sources continuous or cocontinuous functors are mentioned, and no  equivalences of Grothendieck topologies are discussed, we show here  how well everything fits together.    
 
\subsection{Internal categories}

\label{internal-categories}

Internal categories have been introduced by Ehresmann \cite{Ehresmann1963}. We use the following version.

\begin{definition}
\label{internal-category}
A \emph{$\mathscr{C}$-category $X$} consists of objects $X_0,X_1$ of $\mathscr{C}$, universal morphisms $s,t:X_1 \to X_0$ called \emph{source} and \emph{target}, respectively, a morphism $i: X_0 \to X_1$, and a morphism $X_1 \ttimes st X_1 \to X_0: (\gamma',\gamma) \mapsto \gamma' \circ \gamma$ called \emph{composition}, such that the usual axioms of a category are satisfied. A \emph{$\mathscr{C}$-groupoid} is a $\mathscr{C}$-category $X$ together with a morphism $inv:X_1 \to X_1$ satisfying the usual axioms for an inversion w.r.t. composition. If $\mathscr{C}$ has a terminal object $*$, then a groupoid with $X_0=\ast$ is called a \emph{group}.  
\end{definition}

We refer to \cite[\S 3.1]{Meyer2014} for a diagrammatic version of the \quot{usual axioms}. In the definition of an internal category , the crucial point is to guarantee  the existence of the  fibre products 
\begin{equation*}
X_k:=\underbrace{X_1 \ttimes st X_1 \ttimes s {} ... \ttimes {} t X_1}_{k\text{ times}}
\end{equation*}
for at least $k=2$ (the domain of the composition map) and $k=3$ (in order to define associativity). Ehresmann in \cite{Ehresmann1963} and Roberts \cite{Roberts2012} just did this an requires the existence of $X_2$ and $X_3$. Since later on, we have to consider other fibre products along $s$ and $t$, so that it is just more convenient to require -- as we do in \cref{internal-category} -- that $s$ and $t$ are universal, implying the existence of all fibre products $X_k$.

Another option, chosen in many treatments such as  \cite{Meyer2014},  is to  require that $s$ and $t$ are coverings for a Grothendieck topology, also implying the existence of $X_k$. We do not do this for two reasons: first, it is only meaningful for singleton Grothendieck topologies, and second, because it misuses Grothendieck topologies to infer the existence of fibre products, whereas the purpose of Grothendieck topologies is to implement locality.  

Nonetheless, our $\mathscr{C}$-groupoids are precisely the \quot{groupoids in $(\mathscr{C},\,\Uni(T))$} in the sense of \cite[Def. 3.2]{Meyer2014}, for any Grothendieck topology $T$ on $\mathscr{C}$. 
Indeed, source and target maps of our $\mathscr{C}$-groupoids are globally split (via $i:X_0 \to X_1$) and universal. By \cref{locally-split-morphisms:e} they are hence universal $T$-locally split, for any $T$, and thus $\Uni(T)$-coverings. Conversely, if $s$ and $t$ are $\Uni(T)$-coverings, then they are in particular universal.

\begin{remark}
If $\mathscr{F}:\mathscr{C}_1 \to \mathscr{C}_2$ is a functor that preserves  universal morphisms and fibre products with universal morphisms, it sends $\mathscr{C}_1$-categories/ groupoids/ groups to $\mathscr{C}_2$-ones.  
\end{remark}

Similarly to \cref{internal-category}, we define functors and natural transformations \cite[\S 3.5]{Meyer2014};  then, $\mathscr{C}$-categories and $\mathscr{C}$-groupoids form  2-categories \cite[Prop. 3.12]{Meyer2014}, which we denote by $\mathscr{C}$-$\Cat$ and $\mathscr{C}$-$\Grpd$, respectively. 

\begin{definition}
\label{weak-equivalence}
Let $(\mathscr{C},T)$ be a site. A functor $F:G \to H$ between $\mathscr{C}$-groupoids is called \emph{$T$-weak equivalence} if:
\begin{enumerate}[(a)]

\item 
\label{weak-equivalence:a}
it is \emph{fully faithful} in the sense that the diagram
\begin{equation*}
\alxydim{}{G_1 \ar[r]^{F_1} \ar[d]_{(s,t)} & H_1\ar[d]^{(s,t)} \\ G_0 \times G_0 \ar[r]_{F_0 \times F_0} & H_0 \times H_0}
\end{equation*}
is a pullback. 

\item
\label{weak-equivalence:b}
it is \emph{$T$-essentially surjective} in the sense that
the
map
\begin{equation*}
s\circ \pr_2 : G_0 \ttimes{F_0}t H_1 \to H_0 
\end{equation*}
is universal  $T$-locally split. 

\end{enumerate} 
\end{definition}

The following construction is common and will be useful later. If $G$ is a $\mathscr{C}$-groupoid and $\pi:Y \to G_0$ is a universal $T$-locally split morphism, then we consider the $\mathscr{C}$-groupoid $G^{\pi}$ with objects $G^{\pi}_0:=Y$, morphisms $G^{\pi}_1:=Y \ttimes\pi t G_1 \ttimes s\pi Y$, source and target maps $s := \pr_{3}$ and $t := \pr_1$, and composition 
\begin{equation*}
\alxydim{@C=1em}{\mqquad G^{\pi}_1 \ttimes ts G^{\pi}_1 = Y \ttimes\pi t G_1 \ttimes s\pi Y \ttimes \id\id Y \ttimes \pi t G_1 \ttimes s\pi Y  \ar[d]^-{\pr_{1256}} \\ Y \ttimes\pi t G_1  \ttimes st G_1 \ttimes s\pi Y \ar[r]^{\id \times \circ \times \id} & Y \ttimes\pi t G_1 \ttimes s\pi Y=G^{\pi}_1\text{.}}
\end{equation*} 
One can identify $G^{\id}=G$. 
If 
\begin{equation}
\label{refinement-diagram}
\alxydim{}{Y \ar[rr]^{\rho} \ar[dr]_{\pi} && Y' \ar[dl]^{\pi'} \\ & G_0}
\end{equation}
is a commutative diagram with $\pi$ and $\pi'$ universal $T$-locally split, then we obtain a functor $G^{\rho,\pi'}:G^{\pi} \to G^{\pi'}$ with $G^{\rho}_0=\rho$. In particular, for $\pi'=\id_{G_0}$ and $\rho=\pi$, we have a functor $G^{\pi,\id}:G^{\pi} \to G$.

\begin{lemma}
The  functors $G^{\rho,\pi'}$ are  $T$-weak equivalences, for all diagrams as in \cref{refinement-diagram}.
\end{lemma} 

\begin{proof} 
The relevant morphism in \cref{weak-equivalence:b} is
\begin{equation}
\label{erwrweyxc}
\pr_3 : Y \ttimes{\pi} t G_1 \ttimes s{\pi'} Y' \to Y' 
\end{equation}
and we show that it is universal $T$-locally split. It is clearly universal as it is a projection from a fibre product with all morphisms universal. To see that it is $T$-locally split, we choose a covering $(\pi_i: U_i \to G_0)_{i\in I}$ with local sections $\rho_i: U_i \to Y$, using that $\pi$ is $T$-locally split. Then, we consider the pullback covering $(\pi'^{*}\pi_i: \pi'^{*}U_i \to Y')_{i\in I}$. Now, the local sections 
\begin{equation*}
\alxydim{@C=8em}{\pi'^{*}U_i = U_i \times_{G_0} Y' \ar[r]^-{(\rho_i  \circ \pr_1,i \circ \pr_2,\pr_2 )} & Y  \ttimes \pi t G_1 \ttimes s{\pi'} Y' }
\end{equation*}
show that \cref{erwrweyxc} splits over the cover $(\pi_i'^{*}\pi)$. 
\end{proof}

\begin{remark}
If $T$ is local, then our $T$-weak equivalences are exactly the weak equivalences of \cite[Thm. 3.28]{Meyer2014} for the site $(\mathscr{C},\,\Uni(T))$ because $\Uni(T)$ satisfies their Assumption 2.6, see \cref{assumption-2.6}.
Likewise, our definition coincides with Bunge-Paré \cite{Bunge1979} for the site $(\mathscr{C},\,\Uni(T))$, who, however impose more conditions on $\mathscr{C}$ and $T$. 
Roberts \cite[Def. 4.14]{Roberts2012} uses a slightly different definition of $T$-essentially surjectivity for a functor $F:G \to H$: he requires that there exists a (singleton) $T$-covering $\pi:Z \to H_0$ and a functor $K: H^{\pi} \to G$ with a natural isomorphism $F \circ K \cong \pi$. This is strictly stronger then our condition in \cref{weak-equivalence:b}; however, under the additional assumption that $F$ is fully faithful, they coincide, see \cite[Prop. 4.19]{Roberts2012}. In particular, Roberts' $T$-weak equivalence coincide with ours (when $T$ is singleton).  
\end{remark}

\begin{remark}
If $\mathscr{C}$ has a functor $\mathscr{C} \to \Set$ that preserves fibre products,   then every $\mathscr{C}$-groupoid has an underlying ordinary groupoid. If $T$ is subcanonical, and the functor $\mathscr{C} \to \Set$ additionally preserves epimorphisms, then a $T$-weak equivalence  $F$ induces an ordinary equivalence of the underling groupoids, since  $T$-locally split morphisms are epimorphisms (\cref{epimorphisms}), and epimorphisms in $\Set$ are surjective. 
\end{remark}

An immediate consequence of our definitions is the following.

\begin{lemma}
\label{continuous-functors-and-weak-equivalences}
Continuous functors $\mathscr{F}:(\mathscr{C}_1,T_1) \to (\mathscr{C}_2,T_2)$ send $T_1$-weak equivalences between $\mathscr{C}_1$-groupoids to $T_2$-weak equivalences between $\mathscr{C}_2$-groupoids.
\end{lemma}

\subsection{Principal bundles}

Let $\mathscr{C}$ be a category and  $G$ be a $\mathscr{C}$-groupoid.
A \emph{right action} of  $G$ on an object $P\in \mathscr{C}$ is a pair of a morphism $\phi:P \to G_0$ (called \emph{anchor}) and  $\rho: P \ttimes\phi t G_1 \to P$ such that the diagrams
\begin{equation*}
\alxydim{}{P \ttimes\phi t G_1 \ar[r]^-{\rho} \ar[d]_{\pr_2} & P \ar[d]^{\phi} \\ G_1 \ar[r]_{s} & G_0}
\quand
\alxydim{}{P \ttimes\phi t G_1 \ttimes st G_1 \ar[r]^-{\rho \times \id} \ar[d]_-{\id \times \circ} &  P \ttimes\phi tG_1 \ar[d]^{\rho} \\ P\ttimes\phi t G_1 \ar[r]_{\rho} & P}
\end{equation*}
are commutative, see \cite[\S 4]{Meyer2014} or \cite[Def. I.2.13]{Carchedi2011}. Here, all fibre products exist because $s$ and $t$ are universal; we do not put any condition on $\phi$.  A \emph{left action} is a right action of the opposite groupoid.
A morphism $p:P \to X$ to some object $X$ is called \emph{invariant}, if $p\circ \rho = p\circ \pr_1$. A morphism $f:P_1 \to P_2$ between objects with right $G$-actions $(\phi_1,\rho_1)$ and $(\phi_2, \rho_2)$ is called \emph{equivariant}, if the diagrams
\begin{equation*}
\alxydim{}{P_1 \ar[rr]^-{f} \ar[dr]_{\phi_1} && P_2 \ar[dl]^{\phi_2} \\ & G_0}
\quand
\alxydim{}{P_1 \ttimes{\phi_1} t G_1 \ar[d]_{\rho_1} \ar[r]^-{f \times \id} & P_2 \ttimes{\phi_2} t G_1 \ar[d]^{\rho_2} \\ P_1 \ar[r]_-{f} &P_2 }
\end{equation*}
are commutative.

For the following definition of a principal bundle, we follow the philosophy (see, e.g., \cite{Sati}) that there is a \quot{formal} version, independent of any Grothendieck topology, to which a locality condition may be added or not.  

\begin{definition}
\label{principal-G-bundle}
Let $\mathscr{C}$ be a category, $X\in \mathscr{C}$ and $G$ be a $\mathscr{C}$-groupoid.
A  \emph{principal $G$-bundle} over $X$ is an object $P\in \mathscr{C}$ with a right $G$-action $(\phi,\rho)$ and a $G$-invariant, universal morphism $p:P\to X$, such that the \emph{shear map}
\begin{equation*}
(\pr_1,\rho):P \ttimes\phi t G_1 \to P \ttimes pp P
\end{equation*}
is an isomorphism.
 A morphism of principal $G$-bundles $\phi:P\to P'$ is an equivariant morphism such that $p' \circ \phi=p$. Principal $G$-bundles over $X$ form a category $\Bun_G(X)$. 
\end{definition}

That the shear map is an isomorphism means $G$ acts \quot{free and transitively on the fibres}. The category $\Bun_G(X)$ is a actually a groupoid; this can be proved in the usual way.
 That $p$ is universal allows to define the pullback of a principal $G$-bundle along an arbitrary morphism $f:X' \to X$, furnishing a functor
 \begin{equation*}
f^{*}: \Bun_G(X) \to \Bun_G(X')\text{.}
\end{equation*}

\begin{example}
\label{groupoid-as-bundle}
If $G$ is a $\mathscr{C}$-groupoid, then $t:G_1 \to G_0$ together with the right $G$-action defined by $\phi := s: G_1 \to G_0$ and $\rho := \circ: G_1 \ttimes st G_1 \to G_1$ is a principal $G$-bundle. 
\end{example}

\begin{example}
\label{trivial-principal-bundle}
If $G$ is a $\mathscr{C}$-groupoid and $U\in \mathscr{C}$ is an object with a morphism $\psi:U \to G_0$, then $I_{\psi} := U \ttimes\psi t G_1$ carries a right $G$-action similar to that of \cref{groupoid-as-bundle}, with anchor   $\phi := s \circ \pr_2:I_{\psi} \to G_0$. Moreover,  equipped with the morphisms $p:=\pr_1: I_{\psi} \to U$, this defines  a principal $G$-bundle over $U$; it is called the \emph{trivial $G$-bundle over $U$ with anchor $\psi$}. In order to verify this, first note that $p$ is the pullback of $t:G_1 \to G_0$ along $\psi$, hence universal. 
It is easy to see that the diagram
\begin{equation*}
\alxydim{@C=4em}{U\ttimes \psi t G_1 \ttimes st G_1 \ar[r]^-{\pr_1,\circ} \ar[d]_{\pr_{12}} & I_{\psi} \ar[d]^{p} \\  I_{\psi} \ar[r]_{p} &  U}
\end{equation*}
is a fibre product,  and that the shear map is the identity; hence, an isomorphism.
\end{example}

\begin{remark}
\label{sections-and-local-trivializations}
Let  $G$ be a $\mathscr{C}$-groupoid, and $p:P\to X$   be a principal $G$-bundle. Let $\pi:U \to X$ be a morphism. We fix the following standard terminology:
\begin{enumerate}

\item
A \emph{section} over $\pi$ is a morphism $\sigma: U \to P$ such that $p \circ \sigma=\pi$. 

\item
A \emph{trivialization} over $\pi$ is a morphism $\psi: U \to G_0$ and an isomorphism $\Phi: I_{\psi} \to \pi^{*}P$ of principal $G$-bundles over $U$.

\end{enumerate}
For $\pi=\id_X$, we say \emph{global section} and \emph{global trivialization}.
As usual, sections and trivializations over the same morphism $\pi:U \to X$ are in a canonical bijection. Namely, if $\Phi: I_{\psi} \to \pi^{*}P$ is a  trivialization over $\pi:U \to X$, then the morphism
\begin{equation*}
\alxydim{@C=3em}{U  \ar[r]^-{(\id,i \circ \psi)} & U \ttimes\psi t G_1 \ar[r]^{\Phi} & U \ttimes \pi p P \ar[r]^-{\pr_2} & P }
\end{equation*}
is a  section over $\pi:U \to X$, and this assignment establishes the bijection.
\end{remark}

This finalizes the formal part of principal bundles; next, we involve a Grothendieck topology $T$ on $\mathscr{C}$.

\begin{definition} 
\label{locally-trivial-bundle}
Let $(\mathscr{C},T)$ be a site and $G$ be a $\mathscr{C}$-groupoid. A principal $G$-bundle $P$ is called \emph{$T$-locally trivial}, if 
its projection $p:P \to X$ is  $T$-locally split. The full subcategory over all $T$-locally trivial principal $G$-bundles over $X$ is denoted by $\Bun_G(X)^{T}$. 
\end{definition}

Evidently, equivalent Grothendieck topologies have the same locally trivial principal bundles. 
The next result follows directly from the definition and \cref{sections-and-local-trivializations}. 

\begin{proposition}
\label{principal-bundles-equivalence}
The following are equivalent for a principal $G$-bundle  $P$ over $X$:
\begin{enumerate}[(a)]

\item 
$P$ is $T$-locally trivial.

\item
There exists a $T$-covering $(\pi_i: U_i \to X)_{i\in I}$ such that $P$ has a section over each $\pi_i$.

\item
There exists a $T$-covering $(\pi_i: U_i \to X)_{i\in I}$ such that $P$ has a  trivialization over each  $\pi_i$.

\end{enumerate}
\end{proposition}

\begin{example}
\label{examples-locally-trivial}
\begin{itemize}
\item 

Trivial principal $G$-bundles (\cref{trivial-principal-bundle}) are $T$-locally trivial, for any $T$, as their projection has a global section. 
\item
Every principal $G$-bundle is  $T_{dis}$-locally trivial.

\item
A principal $G$-bundle is $T_{indis}$-locally trivial if and only if it has a global section or trivialization.

\item
A principal $G$-bundle is $T_{\mathscr{C}}$-locally trivial if its projection is a universally effective epimorphism.

\item
Our $T$-local principal $G$-bundles are precisely the \quot{$G$-bundles in $(\mathscr{C},\,\Uni(T))$} defined in \cite[Def. 5.1]{Meyer2014}.

\end{itemize}
\end{example}

Finally, the following results follows directly from the definitions.

\begin{lemma}
\label{functoriality-of-G-bundles}
Suppose $\mathscr{F}:(\mathscr{C}_1,T_1) \to (\mathscr{C}_2,T_2)$ is a  continuous functor between sites. If $p: P \to X$ is a $T_1$-locally trivial principal $G$-bundle, then $\mathscr{F}(\pi):\mathscr{F}(P)\to \mathscr{F}(X)$ is a $T_2$-locally trivial principal $\mathscr{F}(G)$-bundle over $\mathscr{F}(X)$.
Thus, we obtain a functor
\begin{equation*}
\mathscr{F}_{*}: \Bun_G(X)^{T_1} \to \Bun_{\mathscr{F}(G)}(\mathscr{F}(X))^{T_2}\text{.}
\end{equation*}
\end{lemma}

\subsection{Bibundles}

\label{bibundles}

We internalize the definition of a bibundle in the standard way. 

\begin{definition}
Let $G$ and $H$ be $\mathscr{C}$-groupoids in a category $\mathscr{C}$. A \emph{$G$-$H$-bibundle}  consists of an object $P\in \mathscr{C}$, a left $G$-action on $P$ and a right $H$-action on $P$, such that the actions commute and the left anchor $P \to G_0$ together with the right action turns $P$ into a principal $H$-bundle over $G_0$. If $\mathscr{C}$ carries a Grothendieck topology $T$, then a bibundle is called \emph{$T$-locally trivial} if its principal $H$-bundle over $G_0$ is $T$-locally trivial. If $P_1,P_2$ are $G$-$H$-bibundles then a \emph{bibundle transformation} is a morphism $\varphi:P_1 \to P_2$ that is equivariant for both actions.
\end{definition}

Our $T$-locally trivial $G$-$H$-bibundles are precisely the \quot{bibundle functors from $G$ to $H$ in $(\mathscr{C},\,\Uni(T))$} of \cite[Def. 6.1]{Meyer2014}.

Obviously, $G$-$H$-bibundles form a groupoid $\Bibun GH$, and the $T$-locally trivial ones form a full subgroupoid $\Bibun GH^{T}$. Since $\Bibun GH=\Bibun GH^{T_{dis}}$ (see \cref{examples-locally-trivial}) it suffices to treat the $T$-locally trivial case in the following. \Cref{functoriality-of-G-bundles} implies that any continuous functor $\mathscr{F}: (\mathscr{C}_1,T_1) \to (\mathscr{C}_2,T_2)$ induces a functor
\begin{equation*}
\mathscr{F}_{*}: \Bibun GH(X)^{T_1} \to \Bibun{\mathscr{F}(G)}{\mathscr{F}(H)}(\mathscr{F}(X))^{T_2}\text{.}
\end{equation*}

\begin{example}
For $G$ a $\mathscr{C}$-groupoid and $X\in \mathscr{C}$, a (locally trivial) bibundle $F:X_{dis} \to G$ is precisely the same as a (locally trivial) principal $G$-bundle over $X$, and we have an equivalence of categories $\Bun_G(X)^{T} \cong \Bibun {X_{dis}}G^{T}$.
\end{example}

\begin{example}
\label{bibundles-from-functors}
If $\Fun(G,H)$ denotes the groupoid of functors from $G$ to $H$, together with natural transformations, then there is a fully faithful functor
\begin{equation*}
\Fun(G,H) \to \Bibun GH^{T}\text{,}
\end{equation*}
no matter what $T$ is. 
It associates to a  functor $F:G \to H$ the bibundle $P_F$ which is, as a principal $H$-bundle over $G_0$, the trivial $H$-bundle with anchor $F_0: G_0 \to H_0$ (see \cref{trivial-principal-bundle}). Thus, its total space is $P_F := G_0 \ttimes F t H_1$, and the left anchor is $\alpha_l=\pr_1: P_F \to G_0$.
The additional left $G$-action is defined by
\begin{equation*}
\alxydim{@C=3em}{ G_1 \ttimes s{\id}G_0 \ttimes F t H_1 \ar[r]^-{\pr_{13}} &  G_1 \ttimes {F \circ s}t H_1 \ar[r]^-{(t,F) \times \id} & G_0 \ttimes Ft H_1 \ttimes st H_1  \ar[r]^-{\id \times \circ} & G_0 \ttimes F t H_1\text{.}}
\end{equation*}
\end{example}
\begin{remark}
\label{bicat-bibundles}
Under an additional assumption on the Grothendieck topology, namely that $T$ is local and that actions of \v Cech groupoids for $\Uni(T)$-coverings are basic, one can define the composition of $T$-locally trivial bibundles \cite[Prop. 7.8]{Meyer2014}. Then, the following results hold:
\begin{enumerate}[(i)]

\item 
\label{bicat-bibundles:a}
$\mathscr{C}$-groupoids, locally trivial bibundles, and bibundle transformations form a bicategory $(\mathscr{C},T)$-$\Grpd^{bi}$ with $\Hom_{(\mathscr{C},T)\text{-}\Grpd^{bi}}(G,H)=\Bibun GH^{T}$ \cite[Thm. 7.13]{Meyer2014}. 

\item
\label{bicat-bibundles:b}
A $G$-$H$-bibundle $P$ is weakly invertible in this bicategory if and only if the right anchor and the left action make it a principal $G^{op}$-bundle over $H_0$, and a weak inverse is obtained by  regarding $P$ as an  $H^{op}$-$G^{op}$-bibundle and then using the canonical isomorphisms $H^{op}\cong H$ and $G^{op}\cong G$ (induced by  inversion); see \cite[Thm. 7.23]{Meyer2014}.   

\item
\label{bicat-bibundles:c}
The construction in \cref{bibundles-from-functors} furnishes a 2-functor from the 2-category of $\mathscr{C}$-groupoids, functors, and natural transformations, such that a  2-functor becomes an invertible bimodule if and only if it is a $T$-weak equivalence \cite[Prop. 6.7]{Meyer2014}. 

\end{enumerate}
\end{remark}

\subsection{Anafunctors}

The following definitions of a $T$-anafunctor and transformations is the traditional one for the site  $(\mathscr{C},\,\Uni(T))$, see \cite{makkai1}, \cite{bartels}, \cite[Def. 5.1]{Roberts2012} \cite[Defs. 3.17 \& 3.19]{Meyer2014} 
\begin{definition}
Let $(\mathscr{C},T)$ be a site and let $G$ and $H$ be $\mathscr{C}$-groupoids. A \emph{$T$-anafunctor} $F: G \to H$ is a pair $F=(\pi,F^{\pi})$ consisting of a universal $T$-locally split morphism $\pi:Y \to G_0$ and a functor $F^{\pi}: G^{\pi} \to H$.
A transformation $F \Rightarrow F'$ between two $T$-anafunctors $F,F': G \to H$ is a natural transformation $\eta: F^{\pi} \circ G^{\rho,\pi} \Rightarrow F^{\pi'} \circ G^{\rho',\pi'}$, where $\rho: Y \ttimes\pi{\pi'}Y' \to Y$   and $\rho': Y \ttimes\pi{\pi'}Y' \to Y'$ are the projections. 
\end{definition}

One can compose transformations
between $T$-anafunctors to obtain a category $\Ana^{T}(G,H)$.

\begin{remark}
If $T$ is a singleton Grothendieck topology, then one can assume without loss of generality that $\pi$ in the definition of a $T$-anafunctor is a $T$-covering. Indeed, if $F=(\pi,F^{\pi})$ is a $T$-anafunctor with $\pi:Y \to G$ universal $T$-locally split, we choose a $T$-covering $\tilde\pi: \tilde Y \to G$ and a local section $\rho: \tilde Y \to Y$. Then, $(\tilde\pi, F^{\pi} \circ G^{\rho,\pi})$ is another $T$-anafunctor and isomorphic to $F$. Via this result, our definition coincides with  \cite[Def. 5.1]{Roberts2012}.\end{remark}

A direct consequence of our definitions is the following result. 

\begin{lemma}
\label{continuous-functors-and-anafunctors}
If $\mathscr{F}:(\mathscr{C}_1,T_1) \to (\mathscr{C}_2,T_2)$ is a continuous functor, and $F=(\pi,F^{\pi}):G \to H$ is a $T_1$-anafunctor, then $\mathscr{F}(F):=(\mathscr{F}(\pi),\mathscr{F}(F^{\pi}))$ is a $T_2$-anafunctor from $\mathscr{F}(G)$ to $\mathscr{F}(H)$.
\end{lemma}

\cite[Prop. 12]{bartels}, \cite[Thm. 5.16]{Roberts2012} and \cite[Thm. 3.22]{Meyer2014} (using $(\mathscr{C},\,\Uni(T))$) all prove the following result, for which the first two references make additional assumptions on $T$, but the third reference avoids any.

\begin{proposition}
\label{bicategory-anafunctor}
Let $(\mathscr{C},T)$ be a site.
$\mathscr{C}$-groupoids, $T$-anafunctors, and transformations form a bicategory $(\mathscr{C},T)$-$\Grpd^{ana}$, with $\Hom_{(\mathscr{C},T)\text{-}\Grpd^{ana}}(G,H)=\Ana^{T}(G,H)$. A $T$-anafunctor $F=(\pi,F^{\pi})$ is weakly invertible if and only if the functor $F^{\pi}$ is a $T$-weak equivalence. 
\end{proposition}

\begin{example}
\label{anafunctors-from-functors}
Any functor $F: G \to H$ induces an $T$-anafunctor $\tilde F=(\id_{G_0},F)$, no matter what $T$ is.
A natural transformation $\eta:F \Rightarrow G$ between functors induces a transformation between the associated $T$-anafunctors $\tilde F$ and $\tilde G$. This defines a fully faithful functor
\begin{equation*}
\Fun(G,H) \to \Ana^{T}(G,H)\text{.}
\end{equation*}
Obviously, it sends $T$-weak equivalences to invertible $T$-anafunctors. Roberts \cite[Prop. 5.17]{Roberts2012} and Meyer-Zhu prove \cite[Thm. 3.23]{Meyer2014} that the corresponding 2-functor
\begin{equation*}
\mathscr{C}\text{-}\Cat^{bi} \to (\mathscr{C},T)\text{-}\Cat^{ana}
\end{equation*} 
is a localization of $\mathscr{C}\text{-}\Cat^{bi}$ at the class of $T$-weak equivalences.
\end{example}

\begin{example}
\label{anafunctors-and-bibundles}
If $P$ is a $T$-locally trivial  $G$-$H$-bibundle, then there is an associated $T$-anafunctor $F_P: G \to H$ with $F_P=(\alpha_l,\tilde P)$, where $\alpha_l: P \to G_0$ is the left anchor of $P$ and $\tilde P:G^{\alpha_l}\to H$ is the functor  defined on objects by the right anchor $\alpha_r: P \to H_0$ and on morphisms by the composite 
\begin{equation*}
\alxydim{}{P \ttimes{\alpha_l}{t} G \ttimes{s}{\alpha_l} P \ar[r]^-{\id \times \rho_{l}} & P \ttimes{\alpha_l}{\alpha_l}P \ar[r]^-{\cong} & P \ttimes{\alpha_l}{t} H \ar[r]^-{\pr_2} & H }
\end{equation*}
This defines a functor
\begin{equation*}
\Ana: \Bibun GH^{T}\to \Ana^{T}(G,H)
\end{equation*}
that is in fact an equivalence of categories \cite[Thm. 7.15]{Meyer2014}. 
The diagram
\begin{equation*}
\alxydim{@C=0em}{& \Fun(G,H) \ar[dr]^{\text{\cref{anafunctors-from-functors}}}\ar[dl]_{\text{\cref{bibundles-from-functors}}} \\ \Bibun GH^{T} \ar[rr]_{\Ana} && \Ana^{T}(G,H)}
\end{equation*}
commutes up to a canonical invertible transformation \cite[Lemma 6.11]{Meyer2014}. Moreover, under the assumptions that allow to form the bicategory $(\mathscr{C},T)$-$\Grpd^{bi}$ (see \cref{bicat-bibundles})  $\Ana$ induces an equivalence of bicategories
 \begin{equation*}
(\mathscr{C},T)\text{-}\Grpd^{bi} \cong (\mathscr{C},T)\text{-}\Grpd^{ana}\text{,} 
\end{equation*}
see \cite[Thm. 7.15]{Meyer2014}.
In particular, it sends invertible bibundles to invertible anafunctors.
\end{example}

\section{Sheaves}

\label{sheaves}

We discuss sheaves w.r.t. Grothendieck topologies (\cref{grothendieck-topology}), and focus on discussing their compatibility with equivalences between Grothendieck topologies (\cref{equivalence-between-grothendieck-topologies}), and functors between sites (\cref{continuous-functors,cocontinuous-functors}). 

\subsection{Descent}

\begin{definition}
A \emph{presheaf} on a category $\mathscr{C}$ is a functor
\begin{equation*}
\mathcal{F}: \mathscr{C}^{op} \to \Set\text{,}
\end{equation*}
where $\Set$ is the category of sets. We denote by $\PSh(\mathscr{C}):=\Fun(\mathscr{C}^{op},\Set)$ the category of presheaves on $\mathscr{C}$.
\end{definition}

\begin{remark}
If $X$ is a topological space, let $\mathscr{C}:= \Open_X$ be the category whose objects are the open sets of $X$, and whose morphisms are all the inclusions $U \incl V$ of open sets. A presheaf on $\Open_X$ is  what is usually defined as a \quot{presheaf on $X$}.   
\end{remark}

If $\mathcal{F}\in \PSh(\mathscr{C})$ and $f:X \to Y$ is a morphism in $\mathscr{C}$, we usually write $f^{*} := \mathcal{F}(f)$ for the map $\mathcal{F}(f): \mathcal{F}(Y) \to \mathcal{F}(X)$ and, we call this the map \emph{induced} from $f$ by the presheaf $\mathcal{F}$.

\begin{definition}
\label{extensive-sheaf}
A presheaf $\mathcal{F}$ is called \emph{extensive}, if it maps coproducts in $\mathscr{C}$ to products in $\Set$.
\end{definition}

More precisely, this means that for all coproducts $X=\displaystyle \coprod_{i\in I}X_i$ in $\mathscr{C}$, the map
\begin{equation}
\label{sheaf-extensive}
\mathcal{F}(X) \to \prod_{i\in I} \mathcal{F}(X_i)
\end{equation}
induced by the coproduct injections $X_i \to X$ and the universal property of the product, is a bijection. 
In particular, if $\mathscr{C}$ has an initial object $\emptyset$, then an extensive presheaf satisfies $\mathcal{F}(\emptyset)=\{*\}$.

\begin{definition}
\label{sheaf}
Let $\mathscr{C}$ be a category and  $\mathcal{F}$ be a presheaf on $\mathscr{C}$. If $\pi:Y \to X$ is a morphism whose double fibre product $Y \times_X Y$ exists, then we say that $\mathcal{F}$ \emph{satisfies descend w.r.t. $\pi$} if the diagram
\begin{equation*}
\alxydim{}{\mathcal{F}(X) \ar[r]^{\pi^{*}} & \mathcal{F}(Y) \ar@{}@<0.2em>[r]^-{\pr_1^{*}} \arr[r]_-{\pr_2^{*}} & \mathcal{F}(Y \times_X Y)}
\end{equation*}
is an equalizer. If $T$ is a  Grothendieck topology on $\mathscr{C}$, we say that $\mathcal{F}$ is a \emph{$T$-sheaf}, if it is extensive and satisfies descend w.r.t. all universal $T$-locally split morphisms. We denote by $\Sh(\mathscr{C},T)$ the full category of $\PSh(C)$ over the $T$-sheaves.
\end{definition}

\noindent
This sightly non-standard definition of a sheaf has two main advantages over the traditional definition of a sheaf, recalled below in \cref{traditional-sheaf}:  
\begin{enumerate}

\item 
It is manifestly invariant under equivalences of Grothendieck topologies; indeed,  $T_1 \prec T_2$ implies via \cref{universal-completion:c} that every $T_2$-sheaf is a $T_1$-sheaf.  

\item
It is suitable for singleton \emph{and} non-singleton Grothendieck topologies, whereas the traditional sheaf definition is problematic for singleton Grothendieck topologies (see \cref{problem-with-traditional}).

\item
For most \emph{non-singleton} Grothendieck topologies (precisely: superextensive and singletonizable) our definition above is equivalent to the traditional one (\cref{comparison-sheaf}).

\end{enumerate}

\medskip

Next we will provide some methods how to check whether a presheaf is a sheaf. First we note that, since the  diagram in \cref{sheaf} is a diagram in the category $\Set$, its equalizer has the following explicit model, the set \begin{equation*}
\des_{\mathcal{F}}(\pi) := \{ a\in \mathcal{F}(Y) \sep \pr_2^{*}a=\pr_1^{*}a \} \subset \mathcal{F}(Y)
\end{equation*}
of \quot{descent data}.
The map $\pi^{*}: \mathcal{F}(X) \to \mathcal{F}(Y)$ lands in this subset, and the condition that it is an equalizer is thus the following.

\begin{lemma}
Let $\pi:Y \to X$ be a  morphism whose double fibre product $Y \times_X Y$ exists. A presheaf $\mathcal{F}$ satisfies descend w.r.t. $\pi$ if and only if the map 
\begin{equation*}
\pi^{*}: \mathcal{F}(X) \to \des_{\mathcal{F}}(\pi)
\end{equation*}
is a bijection.
\end{lemma}

Next we look at a  \emph{singleton} Grothendieck topology, and where it suffices to check descent for  $T$-coverings instead as for all  universal $T$-locally split morphisms. 

\begin{lemma}
If $T$ is a singleton Grothendieck topology, then $\mathcal{F}$ satisfies descent w.r.t.  all universal $T$-locally split maps if and only if it satisfies descent w.r.t. all $T$-coverings. 
\end{lemma}

\begin{proof}
This is left as an elementary exercise. (A proof can be obtained as a corollary from  \cref{comparison-sheaf} below, and its equivalence \cref{condition:sheaf*} $\Leftrightarrow$ \cref{condition:extensive-traditional-sheaf*}; using that every singleton Grothendieck topology is singletonizable).  
\end{proof}

For non-singleton Grothendieck topologies, we may use the traditional sheaf condition in order to check descent. 

\begin{remark}
\label{traditional-sheaf}
The traditional condition imposed for a presheaf $\mathcal{F}$ in order to be a sheaf w.r.t. to a Grothendieck topology $T$  is that for each  $T$-covering $(\pi_i:U_i \to X)_{i\in I}$ the diagram 
\begin{equation*}
\alxydim{}{\mathcal{F}(X) \ar[r]^-{\prod \pi_i^{*}} & \displaystyle\prod_{i\in I}\mathcal{F}(U_i) \arr[r] & \displaystyle\prod_{i,j\in I}\mathcal{F}(U_i \times_X U_j)}
\end{equation*}
is an equalizer. We refer to presheaves satisfying this condition for all $T$-coverings  as \emph{traditional $T$-sheaves}. 
\end{remark}

\begin{lemma}
\label{descent-for-non-singleton}
Let $T$ be any Grothendieck topology. If $\mathcal{F}$ is a traditional $T$-sheaf, then $\mathcal{F}$ satisfies descent w.r.t. all universal T-locally split maps. 
\end{lemma}

\begin{proof}
Let $\pi:Y \to X$ be a universal $T$-locally split morphism, let $(\pi_i: U_i \to X)_{i\in I}$ be a $T$-covering, and let $\rho_i: U_i \to Y$ be local sections. We obtain a commutative diagram
\begin{equation*}
\alxydim{@R=0.7em}{& \mathcal{F}(Y) \ar[dd]^{\prod \rho_i^{*}} \arr[r] & \mathcal{F}(Y \times_X Y) \ar[dd]^{\prod \rho_i^{*} \times \rho_j^{*}} \\ \mathcal{F}(X) \ar[ur]^{\pi^{*}} \ar[dr]_-{\prod \pi_i^{*}} \\ & \displaystyle\prod_{i\in I} \mathcal{F}(U_i) \arr[r] & \displaystyle \prod_{i,j\in I} \mathcal{F}(U_i \times_X U_j)\text{,}}
\end{equation*}
in which the bottom line is an equalizer by assumption. The factorization
\begin{equation*}
\prod \pi_i^{*} = \prod \rho_i^{*} \circ \pi^{*}
\end{equation*}
shows that $\pi^{*}$ is injective. For surjectivity, we assume that $a\in \mathcal{F}(Y)$ satisfies $\pr_1^{*}a=\pr_2^{*}a$. We have for $i,j\in I$ the following equalities for elements of $\mathcal{F}(U_i \times_X U_j)$:
\begin{equation*}
\pr_1^{*}a_i=\pr_1^{*}\rho_i^{*}a=\rho_i^{*}\pr_1^{*}a=(\rho_i \times\rho_j)^{*}\pr_1^{*}a=(\rho_i \times\rho_j)^{*}\pr_2^{*}a=\rho_j^{*}\pr_2^{*}a=\pr_2^{*}\rho_j^{*}a=\pr_2^{*}a_j
\end{equation*}
This shows that there exists $b\in \mathcal{F}(X)$ such that $\pi_i^{*}b=a_i$. We claim that $\pi^{*}b=a$; this shows that $\mathcal{F}$ satisfies descent w.r.t. to universal $T$-locally split morphisms.
In order to prove that claim, we consider the pullback diagrams
\begin{equation*}
\alxydim{}{Y \times_X U_i \ar[r]^-{\rho_i'} \ar[d]_{\pi'_i} & U_i \ar[d]^{\pi_i} \\ Y \ar[r]_{\pi} & X}
\end{equation*}
in which $(\pi'_i: Y \times_X U_i \to Y)_{i\in I}$ is a $T$-covering. We have, for each $i\in I$,  a commutative diagram
\begin{equation*}
\alxydim{@C=3em}{Y \ar@{=}[d] & Y \times_X U_i \ar[d]^{\id \times \rho_i} \ar[l]_-{\pi_i'^{*}} \ar[r]^-{\rho_i'} & U_i \ar[d]^{\rho_i} \\ Y & Y \times_X Y \ar[l]^-{\pr_1} \ar[r]_-{\pr_2} & Y.}
\end{equation*}
Now we calculate
\begin{equation*}
\pi_i^{\prime*}\pi^{*}b =\rho_i^{\prime*}\pi_i^{*}b=\rho_i^{\prime*}a_i=\rho_i^{\prime *}\rho_i^{*}a=(\id \times \rho_i)^{*}\pr_2^{*}a=(\id \times \rho_i)^{*}\pr_1^{*}a=\pi_i^{\prime*}a\text{.}
\end{equation*}
Since $\mathcal{F}$ is a traditional $T$-sheaf and $\pi_i'$ is a $T$-covering, this implies that $\pi^{*}b=a$.
\end{proof}

The main difference between our definition of a sheaf and the traditional one is that we demand explicitly that sheaves are extensive. The traditional sheaf condition includes this only when the Grothendieck topology is superextensive. We briefly recall this well-known statement.

\begin{lemma}
\label{traditional-sheaf-extensive}
Consider the extensive Grothendieck topology $T^{ext}_{\mathscr{C}}$ on an extensive category $\mathscr{C}$; see \cref{extensive-topology}. A presheaf $\mathcal{F}$ is extensive if and only if it is a traditional $T^{ext}_{\mathscr{C}}$-sheaf. In particular, every traditional sheaf on a superextensive site is extensive. 
\end{lemma}

\begin{proof}
In $T_{\mathscr{\mathscr{C}}}^{ext}$, the empty family is a covering of the initial object, and the traditional sheaf condition implies $\mathcal{F}(\emptyset)=*$. 
If
 $(X_i)_{i\in I}$ is a family of objects $X_i\in \mathscr{C}$ whose coproduct $X$ exists, then  $(\iota_i:X_i \to X)_{i\in I}$ is a $T^{ext}_{\mathscr{C}}$-covering family. Since $\mathcal{F}(\emptyset)=*$ and coproducts are disjoint, we have
\begin{equation*}
\prod_{i,j\in I}\mathcal{F}(X_i \times_{\coprod X_i} X_j) = \prod_{i\in I}\mathcal{F}(X_i \times_{\coprod X_i} X_i)\cong \prod_{i\in I} \mathcal{F}(X_i)\text{,}
\end{equation*}
where the latter bijection is induced by the diagonal maps $\Delta_i: X_i \to X_i \times_{\coprod X_i} X_i$, which are isomorphisms. Thus, the traditional sheaf condition is equivalent to the assertion that
\begin{equation}
\label{sfsdfsdfsd}
\mathcal{F}(X) \to \prod_{i\in I} \mathcal{F}(X_i)
\end{equation}
is a bijection; i.e., that $\mathcal{F}$ is extensive.  
\end{proof}

The following result summarizes the lemmas above.

\begin{proposition}
\label{comparison-sheaf}
Let $T$ be a Grothendieck topology on a category $\mathscr{C}$.
Consider the following three assertions:
\begin{enumerate}[(i)]

\item 
\label{condition:sheaf}
$\mathcal{F}$ is a $T$-sheaf.

\item
\label{condition:extensive-traditional-sheaf}
$\mathcal{F}$ is extensive and a traditional $T$-sheaf.  

\item
\label{condition:traditional-sheaf}
$\mathcal{F}$ is a traditional $T$-sheaf.

\end{enumerate}
Then, the implications \cref{condition:extensive-traditional-sheaf*} $\Rightarrow$ \cref{condition:sheaf*} and \cref{condition:extensive-traditional-sheaf*} $\Rightarrow$ \cref{condition:traditional-sheaf*} hold. If $\mathscr{C}$ is extensive and $T$ is superextensive, then  \cref{condition:extensive-traditional-sheaf*} $\Leftrightarrow$ \cref{condition:traditional-sheaf*}. If $T$ is singletonizable, then  \cref{condition:sheaf*} $\Leftrightarrow\cref{condition:extensive-traditional-sheaf*}$. In particular, if  $\mathscr{C}$ is extensive and $T$ is superextensive and singletonizable, all three assertions are equivalent. 
\end{proposition}

\begin{proof}
The implication \cref{condition:extensive-traditional-sheaf*} $\Rightarrow$ \cref{condition:traditional-sheaf*} is trivial. The implication \cref{condition:extensive-traditional-sheaf*} $\Rightarrow$ \cref{condition:sheaf*} follows from \cref{descent-for-non-singleton}. 
  The implication \cref{condition:traditional-sheaf*} $\Rightarrow$ \cref{condition:extensive-traditional-sheaf*} is \cref{traditional-sheaf-extensive}. Finally, in order to see the implication \cref{condition:sheaf*} $\Rightarrow$ \cref{condition:extensive-traditional-sheaf*} let $(\pi_i: U_i \to X)$ be a $T$-covering. Since $T$ is singletonizable, we may consider $\pi: \coprod U_i \to X$, which is universal  $T$-locally split. Hence, $\mathcal{F}$ satisfies descent w.r.t. $\pi$. Since $\mathcal{F}$ is extensive, this is equivalent to the traditional sheaf  condition for the covering $(\pi_i)_{i\in I}$.   
\end{proof}

\begin{remark}
\label{problem-with-traditional}
If $T$ is singletonizable (but not singleton), then a traditional $\Sing(T)$-sheaf $\mathcal{F}$ is not necessarily a traditional  $T$-sheaf. For example, let $\Z_{\mathscr{C}}$ be the constant presheaf on $\mathscr{C}$ that assigns the set $\Z$ to every object and the identity map $\id_\Z$ to every morphism. This is a traditional sheaf for \emph{every} singleton Grothendieck topology on $\mathscr{C}$, because the relevant  diagram is
\begin{equation*}
\alxydim{}{\Z \ar[r] & \Z \arr[r] & \Z,}
\end{equation*}
which is an equalizer. In particular, $\Z_{\mathscr{C}}$ is a traditional $\Sing(T)$-sheaf. However, the constant presheaf $\Z_{\mathscr{C}}$ is typically not a traditional sheaf for the non-singleton Grothendieck topology $T$, see \cref{traditional-sheaf-bad} for an example.  Our modified definition of a sheaf resolves this: since $T \sim \Sing(T)$, the $T$-sheaves are precisely the $\Sing(T)$-sheaves. 
\end{remark}

If $X\in \mathscr{C}$, then $\Yo_X \in \PSh(\mathscr{C})$ is the presheaf with $\Yo_X(U) := \mathrm{Hom}_{\mathscr{C}}(U,X)$ and $f^{*}: \Yo_X(V) \to \Yo_X(U)$ given by $g \mapsto g \circ f$, for any $f:U \to V$. A presheaf $\mathcal{F}$ is called \emph{representable}, if there exists $X\in \mathscr{C}$ and an isomorphism $\mathcal{F} \cong \Yo_X$.  Because the functor $\mathrm{Hom}_{\mathscr{C}}(-,X): \mathscr{C}^{op} \to \mathscr{C}$ preserves limits, every representable presheaf is extensive. Let $\pi:Y \to X$ be a $T_{\mathscr{C}}$-covering; in particular, a universal effective epimorphism. Since again  $\mathrm{Hom}_{\mathscr{C}}(-,X): \mathscr{C}^{op} \to \mathscr{C}$  preserves limits, it follows that every representable presheaf satisfies descent w.r.t. $\pi$. Thus, every representable presheaf is a $T_{\mathscr{C}}$-sheaf. 

\begin{proposition}
\label{subcanonical-and-sheaves}
A Grothendieck topology $T$ is subcanonical  if and only if every representable presheaf is a $T$-sheaf.
\end{proposition}

\begin{proof}If $T$ is subcanonical, then because $T \prec T_{\mathscr{C}}$, every representable presheaf is  a $T$-sheaf. 
Conversely, assume that $T$ is some  Grothendieck topology such that every representable presheaf is a $T$-sheaf. We claim that every universal  $T$-locally split morphism  $\pi:Y \to X$ is a universally effective epimorphism. This shows that $\Uni(T) \subset T_{\mathscr{C}}$, which implies that $T$ is subcanonical. 
 
 In order to prove the claim, we show that
\begin{equation*}
\alxydim{}{Y \times_X Y \arr[r] & Y \ar[r]^{\pi} & X}
\end{equation*}
is a coequalizer. Indeed, let $f: Y \to Z$ be a cocone, i.e., $f \circ \pr_1 = f \circ \pr_2$ as morphisms $Y \times_X Y \to Z$. Let $\Yo_Z$ be the presheaf represented by $Z$. By assumption, $\Yo_Z$ is a $T$-sheaf. The descent condition w.r.t. $\pi$ says that
\begin{equation*}
\alxydim{}{\Yo_Z(X) \ar[r]^{\pi^{*}} & \Yo_Z(Y) \arr[r] & \Yo_Z(Y \times_X Y) }
\end{equation*}
is an equalizer. Evaluating this for $f\in \Yo_Z(Y)$, we see that there exists a unique $g: X \to Z$ such that $g \circ \pi = f$. This proves the claim. 
\end{proof}

\subsection{Base change}

\label{sheaves-II}

In this section we discuss the behaviour of sheaves under functors between sites. We start by looking only at presheaves. Every functor $\mathscr{F}:\mathscr{C}_1 \to \mathscr{C}_2$ induces a functor 
\begin{equation*}
\mathscr{F}^{*}:\PSh(\mathscr{C}_1) \to \PSh(\mathscr{C}_2)
\end{equation*}
by setting $\mathscr{F}^{*}\mathcal{F} := \mathcal{F} \circ \mathscr{F}$. The functor $\mathscr{F}^{*}$ has both a left adjoint $\mathscr{F}_{!}$ and a right adjoint $\mathscr{F}_{*}$. If $\mathcal{F}$ is a presheaf on $\mathscr{C}_1$, then $\mathscr{F}_{!}\mathcal{F}$ and $\mathscr{F}_{*}\mathcal{F}$ are called the left and right \emph{Kan extensions}, respectively. 
We are primarily interested in the case where $\mathscr{F}^{*}$ is an equivalence of categories, in which case left and right Kan extensions are both essential inverses and hence naturally isomorphic to each other. Thus, it suffices to concentrate on the right Kan extension, which we now describe, following standard sheaf theory, e.g. \cite[\S 7.19]{stacks-project}. 

For $Y$ an object of $\mathscr{C}_2$, we consider the slice category $\mathscr{F}/Y$, whose objects are pairs $(X,\psi)$ with $X$ an object of $\mathscr{C}_1$ and $\psi: \mathscr{F}(X) \to Y$ a morphism in $\mathscr{C}_2$, and where a morphism $(X_1,\psi_1) \to (X_2,\psi_2)$ is a morphism $f:X_1 \to X_2$ in $\mathscr{C}_1$ such that $\psi_2\circ \mathscr{F}(f)=\psi_1$. We have the projection functor $\mathscr{P}_Y:\mathscr{F}/Y \to \mathscr{C}_1: (X,\psi) \mapsto X$ and hence obtain a presheaf $\mathscr{P}_Y^{*}\mathcal{F}$ on $\mathscr{F}/Y$. We define the set
\begin{equation}
\label{right-Kan-extension}
(\mathscr{F}_{*}\mathcal{F})(Y) := \lim \mathscr{P}_Y^{*}\mathcal{F}\text{.} 
\end{equation} 
Any morphism $g: Y_1 \to Y_2$ induces a functor $\mathscr{F}/g: \mathscr{F}/Y_1 \to \mathscr{F}/Y_2$ such that $\mathscr{P}_{Y_2} \circ \mathscr{F}/g = \mathscr{P}_{Y_1}$.
Thus, it induces a map $g^{*}:\lim \mathscr{P}_{Y_2}^{*}\mathcal{F} \to \lim \mathscr{P}_{Y_1}^{*}\mathcal{F}$, turning $\mathscr{F}_{*}\mathcal{F}$ into a presheaf on $\mathscr{C}_2$.
Similarly, one can turn $\mathscr{F}_{*}$ into a functor $\mathscr{F}_{*}:\PSh(\mathscr{C}_1) \to \PSh(\mathscr{C}_2)$.

We may realize the limit \cref{right-Kan-extension} in the category of sets explicitly. An element in $(\mathscr{F}_{*}\mathcal{F})(Y)$ is a family $\{s_{X,\psi}\}_{(X,\psi)\in \mathscr{F}/Y}$ of elements $s_{X,\psi} \in \mathcal{F}(X)$ such that, for all morphisms $f:(X_1,\psi_1) \to (X_2,\psi_2)$ in $\mathscr{F}/Y$ we have 
$f^{*}s_{X_2,\psi_2}=s_{X_1,\psi_1}$.
A morphism $g:Y_1 \to Y_2$ induces the map $g^{*}$ that sends a family $\{s_{X,\psi}\}_{(X,\psi)\in \mathscr{F}/Y_2}\in (\mathscr{F}_{*}\mathcal{F})(Y_2)$ to the family $\{s_{X,g \circ \phi}\}_{(X,\phi)\in \mathscr{F}/Y_1}\in (\mathscr{F}_{*}\mathcal{F})(Y_1)$.
Finally,  if $\phi: \mathcal{F}\to \mathcal{G}$ is a morphism in $\PSh(\mathscr{C}_1)$, then 
\begin{equation*}
(\mathscr{F}_{*}\phi)_Y: (\mathscr{F}_{*}\mathcal{F})(Y) \to (\mathscr{F}_{*}\mathcal{G})(Y)
\end{equation*}
is given by $\{s_{X,\psi}\}_{_{(X,\psi)\in \mathscr{F}/Y}}\mapsto \{\phi_X(s_{X,\psi})\}_{_{(X,\psi)\in \mathscr{F}/Y}}$.

We will first recall how the adjunction between $\mathscr{F}^{*}$ and $\mathscr{F}_{*}$ is established. For each presheaf $\mathcal{F}\in \PSh(\mathscr{C}_1)$, there is a canonical presheaf morphism 
\begin{equation*}
\varepsilon_{\mathcal{F}} : \mathscr{F}^{*}(\mathscr{F}_{*}\mathcal{F}) \to \mathcal{F\text{,}}
\end{equation*}
providing the components of a natural transformation $\varepsilon: \mathscr{F}^{*}\circ \mathscr{F}_{*}\Rightarrow \id_{\PSh(\mathscr{C}_1)}$, see \cite[Lem. 7.19.1]{stacks-project}. Indeed, on the left hand side we have, for $X_0\in \mathscr{C}_1$,  $\mathscr{F}^{*}(\mathscr{F}_{*}\mathcal{F})(X_0)=\mathscr{F}_{*}\mathcal{F}(\mathscr{F}(X_0))$, of which an element is a family $\{s_{X,\psi}\}_{(X,\psi)\in \mathscr{F}/{\mathscr{F}(X_0)}}$. We send this to $s_{X_0,\id_{\mathscr{F}(X_0)}}\in \mathcal{F}(X_0)$, i.e. to the evaluation of the family at the object $(X_0,\id_{\mathscr{F}(X_0)})$ of $\mathscr{F}/{\mathscr{F}(X_0)}$. Basically because the object $(X_0,\id_{\mathscr{F}(X_0)})$ is terminal in $\mathscr{F}/\mathscr{F}(X_0)$,  the following result holds.

\begin{lemma}
\label{epsilon-iso}
$\varepsilon_{\mathcal{F}}$ is an isomorphism if $\mathscr{F}$ is full and faithful. 
\end{lemma}

Another presheaf morphism is
\begin{equation*}
\eta_{\mathcal{G}}: \mathcal{G}\to \mathscr{F}_{*}(\mathscr{F}^{*}\mathcal{G})
\end{equation*}
where $\mathcal{G}$ is a sheaf on $\mathscr{C}_2$, and it sends $s\in \mathcal{G}(Y)$ to the family $\{s_{X,\psi}\}_{(X,\psi)\in \mathscr{F}/Y}$ defined by $s_{X,\psi} := \psi^{*}s \in \mathcal{G}(\mathscr{F}(X))$. This provides the components of a natural transformation $\eta: \id_{\PSh(\mathscr{C}_2)} \Rightarrow \mathscr{F}_{*}\circ \mathscr{F}^{*}$. 

The following result is standard, see e.g.
\cite[Lemma 7.19.2]{stacks-project}.
\begin{proposition}
\label{adjunction-presheaves}
For $\mathscr{F}:\mathscr{C}_1 \to \mathscr{C}_2$ any functor, 
the presheaf morphisms $\varepsilon$ and $\eta$ are, respectively, the counit and the unit of an adjunction
\begin{equation*}
\mathscr{F}^{*}:\PSh(\mathscr{C}_2) \leftrightarrows\ \PSh(\mathscr{C}_1): \mathscr{F}_{*}\text{.} 
\end{equation*}
\end{proposition}

In the following we investigate circumstances under which one can restrict the adjunction of \cref{adjunction-presheaves} to sheaves. 
First we have the following standard result; see, e.g. \cite[Lemma 7.13.2]{stacks-project}, which has to be adapted slightly to our modified definition of a sheaf.

\begin{proposition}
Suppose $\mathscr{F}:(\mathscr{C}_1,T_1) \to (\mathscr{C}_2,T_2)$ is a continuous  and coproduct-preserving functor between sites. If $\mathcal{F}$ is a $T_2$-sheaf on $\mathscr{C}_2$, then $\mathscr{F}^{*}\mathcal{F} $ is a $T_1$-sheaf on $\mathscr{C}_1$.
\end{proposition}

\begin{proof}
First, since  $\mathscr{F}$ preserves coproducts,  $\mathscr{F}^{*}$ sends extensive presheaves to extensive ones.
Second, consider a universal $T_1$-locally split morphism $\pi:Y \to X$. Since $\mathscr{F}$ is continuous,  $\mathscr{F}(\pi):\mathscr{F}(Y) \to \mathscr{F}(X)$ is universal $T_2$-locally split, and  
\begin{equation}
\label{continuous-fibre-product}
\mathscr{F}(Y \times_X Y) \cong \mathscr{F}(Y) \ttimes{\mathscr{F}(\pi)}{\mathscr{F}(\pi)} \mathscr{F}(Y)\text{.}
\end{equation}
That $\mathscr{F}^{*}\mathcal{F}$ satisfies descent w.r.t. $\pi$ means that
\begin{equation*}
\alxydim{}{\displaystyle \mathscr{F}^{*}\mathcal{F}(X) \ar[r]^-{\pi^{*}} &  \mathscr{F}^{*}\mathcal{F}(Y) \arr[r] &  \mathscr{F}^{*}\mathcal{F}(Y \times_X Y)}
\end{equation*}
is an equalizer. Using \cref{continuous-fibre-product} this is exactly the condition that $\mathcal{F}$ satisfies descent w.r.t.  $\mathscr{F}(\pi):\mathscr{F}(Y) \to \mathscr{F}(X)$, which is true since $\mathcal{F}$ is a $T_2$-sheaf.
\end{proof}

For the right Kan extension $\mathscr{F}_{*}$, the situation is slightly more involved. For sheaves in the traditional sense of \cref{traditional-sheaf}, it is a classical result that $\mathscr{F}_{*}$ sends sheaves to sheaves when $\mathscr{F}$ is cocontinuous in the traditional sense. We have to deal with our additional assumption that sheaves must be extensive. We introduce the following terminology.

\begin{definition}
\label{coproduct-inverting}
A functor $\mathscr{F}: \mathscr{C}_1 \to \mathscr{C}_2$  is said to \emph{invert coproducts}, if it is fully faithful, 
and it is essentially surjective on the components of coproducts in its image, i.e.  for every $X\in \mathscr{C}_1$ such that $\mathscr{F}(X)=\coprod_{i\in I} Z_i$, there exist objects $X_i\in \mathscr{C}_1$ and isomorphisms $\gamma_i: \mathscr{F}(X_i)\to Z_i$. 
\end{definition}

\begin{remark}
\label{remark-about-inverting-coproducts}
If $\mathscr{F}$ inverts coproducts, if $\mathscr{F}(X)=\coprod_{i\in I} Z_i$ by means of coproduct injections $\iota_i: Z_i \to \mathscr{F}(X)$, and if we have chosen objects $X_i\in \mathscr{C}$ and isomorphisms $\gamma_i: \mathscr{F}(X_i)\to Z_i$, then there exist unique morphisms $\phi_i:X_i \to X$ such that $\mathscr{F}(\phi_i) = \iota_i \circ \gamma_i$ and $X=\coprod_{i\in I} X_i$ with coproduct injections $\phi_i$.
Indeed, $\phi_i$ exist because $\mathscr{F}$ is fully faithful, and they are coproduct injections because $\mathscr{F}$ reflects colimits (again because it is fully faithful). 
\end{remark}

\begin{proposition}
\label{coproduct-inverting-functors-preserve-extensiveness}
Suppose $\mathscr{F}:\mathscr{C}_1 \to \mathscr{C}_2$ inverts coproducts and $\mathscr{C}_2$ is extensive. If $\mathcal{F}$ is an extensive presheaf on $\mathscr{C}_1$, then $\mathscr{F}_{*}\mathcal{F}$ is extensive. 
\end{proposition}

\begin{proof}
We suppose that $Y=\coprod_{i\in I} Y_i$ is a coproduct in $\mathscr{C}_2$. The canonical map
\begin{equation*}
\alpha:\mathscr{F}_{*}\mathcal{F}(Y) \to \prod_{i\in I} \mathscr{F}_{*}\mathcal{F}(Y_i)
\end{equation*}
sends $s=\{s_{X,\psi}\}_{(X,\psi)\in \mathscr{F}/Y}\in \mathscr{F}_{*}\mathcal{F}(Y)$ to $\alpha(s):=(\alpha_i(s))_{i\in I}$ with 
\begin{equation*}
\alpha_i(s):=\{s_{\tilde X,\iota_i \circ \tilde \psi}\}_{(\tilde X,\tilde \psi)\in \mathscr{F}/{Y_i}}\in \mathscr{F}_{*}\mathcal{F}(Y_i)\text{.}
\end{equation*}
We have to show that $\alpha$ is bijective.

For injectivity, let us assume that $s,s'\in  \mathscr{F}_{*}\mathcal{F}(Y)$ such that $\alpha_i(s)=\alpha_i(s')$ for all $i\in I$. Now let $(X,\psi)\in \mathscr{F}/Y$.
Since $\mathscr{C}_2$ is extensive, one can form, for each $i\in I$, the pullback diagram 
\begin{equation*}
\alxydim{}{Z_i:=\mathscr{F}(X) \times_Y Y_i \ar[r]^-{p_i}\ar[d]_{\kappa_i} & Y_i \ar[d]^{\iota_i} \\ \mathscr{F}(X) \ar[r]_{\psi} & Y }
\end{equation*} 
in which $\kappa_i$ are coproduct injections. \ Now, because $\mathscr{F}$ inverts coproducts, there exist objects $X_i\in \mathscr{C}_1$ and isomorphisms $\gamma_i: \mathscr{F}(X_i) \to Z_i$, which induce via \cref{remark-about-inverting-coproducts} coproduct injections $\phi_i: X_i \to X$; moreover, we obtain morphisms $\psi_i := p_i \circ \gamma_i: \mathscr{F}(X_i) \to Y_i$ such that the diagram
\begin{equation}
\label{klhjslfsd}
\alxydim{}{\mathscr{    F}(X_i) \ar[d]_{\mathscr{F}(\phi_i)} \ar[r]^-{\psi_i} & Y_i \ar[d]^{\iota_i}\\
\mathscr{F}(X) \ar[r]_-{\psi} & Y}
\end{equation}
is commutative. This shows that  $\phi_i: (X_i,\iota_i\circ \psi_i) \to (X,\psi)$ is a morphism in $\mathscr{F}/Y$, and hence, that
$s_{X,\psi} = \phi_i^{*}s_{X_i,\iota_i \circ \psi_i}$;
similarly for $s'$. By assumption, we have $s_{X_i,\iota_i \circ \psi_i}=s'_{X_i,\iota_i \circ \psi_i}$. This shows that $s=s'$.

For surjectivity, we let $t=(t_i)_{i\in I}$, with $t_i=\{t_{\tilde X,\tilde \psi}\}_{(\tilde X,\tilde\psi)\in \mathscr{I}_{Y_i}}\in \mathscr{F}_{*}\mathcal{F}(Y_i)$, i.e., $t_{\tilde X,\tilde \psi}\in \mathcal{F}(\tilde X)$. 
Let $(X,\psi)\in \mathscr{F}/Y$, so that we have a morphism $\psi: \mathscr{F}(X) \to Y$. Again, we choose for each $i\in I$ an object $X_i$ and an isomorphism $\gamma_i$, and induce $\phi_i$ and $\psi_i$ as described above, resulting in the commutative diagram \cref{klhjslfsd}. Then, $(X_i,\psi_i)\in \mathscr{F}/{Y_i}$, and we may consider the family $(t_{X_i,\psi_i})_{i\in I} \in \prod_{i\in I} \mathcal{F}(X_i)$. Its unique preimage under  the  map
\begin{equation*}
 \mathcal{F}(X) \to  \prod_{i\in I}\mathcal{F}(X_i): a \mapsto (\phi_i^{*}a)_{i\in I} \text{,}
\end{equation*}
which is a bijection  because $\mathcal{F}$ is extensive, defines an element $\tilde t_{X,\psi}\in \mathcal{F}(X)$, i.e., $\phi_i^{*}\tilde t_{X,\psi} =t_{X_i,\psi_i}$. One can show that $\tilde t_{X,\psi}$  does not depend on the choices of the objects $X_i$ and the morphisms $\gamma_i$.
Similarly, one can show that  $s:=\{\tilde t_{X,\psi}\}_{(X,\psi)\in \mathscr{F}/Y}$ is an element in $\mathscr{F}_{*}\mathcal{F}(Y)$. 
We claim that $s$ is a preimage of $t$ under the map $\alpha$. To see this, we consider an object $(\tilde X,\tilde \psi)\in \mathscr{F}/{Y_i}$ for some fixed index $i\in I$. We consider $\psi := \iota_i \circ \tilde\psi$; then, $(\tilde X,\psi)$ is an object in $\mathscr{F}/Y$. We may now assume that the choices $X_i$ and $\gamma_i$ are such that $Z_i = \mathscr{F}(\tilde X)$, $\gamma_i=\id$, and $X_i=\tilde X$. Note that the corresponding morphisms are $\phi_i=\id$ and $\psi_i=\tilde\psi$; hence,
\begin{equation*}
\tilde t_{\tilde X,\psi}=\phi_i^{*}\tilde t_{X,\psi}=t_{X_i,\psi_i}=t_{\tilde X,\tilde \psi}\text{.}
\end{equation*}
This shows that
\begin{equation*}
\alpha_i(s)=\{\tilde t_{\tilde X,\iota_i \circ \tilde \psi}\}_{(\tilde X,\tilde \psi)\in \mathscr{I}_{Y_i}} =\{t_{\tilde X,\tilde \psi}\}_{(\tilde X,\tilde \psi)\in \mathscr{I}_{Y_i}}= t_i\text{.}
\end{equation*}
This completes the proof. 
\end{proof}

Via \cref{comparison-sheaf}, this can be used to deduce the following from the classical theory.

\begin{theorem}
\label{extension-preserves-sheaf-1}
Suppose $\mathscr{F}:(\mathscr{C}_1,T_1) \to (\mathscr{C}_2,T_2)$ is a cocontinuous and coproduct-inverting functor between sites. If $\mathcal{F}$ is a $T_1$-sheaf on $\mathscr{C}_1$, then $\mathscr{F}_{*}\mathcal{F}$ is a $T_2$-sheaf on $\mathscr{C}_2$.  
\end{theorem}

\begin{proof}
Since $\mathcal{F}$ is a $T_1$-sheaf, it is also a $\Uni(T_1)$-sheaf, and since $\Uni(T_1)$ is singleton, a traditional $\Uni(T_1)$-sheaf by \cref{comparison-sheaf}. Moreover, that $\mathcal{F}$ is cocontinuous implies that it is cocontinuous in the traditional sense for $\Uni(T_1)$ and $\Uni(T_2)$. Thus,  the classical result (e.g., \cite[Lemma 7.20.2]{stacks-project}) shows that $\mathscr{F}_{*}\mathcal{F}$ is a traditional $\Uni(T_2)$-sheaf. Because $\mathscr{F}$ inverts coproducts, $\mathscr{F}_{*}\mathcal{F}$ is extensive by \cref{coproduct-inverting-functors-preserve-extensiveness}. Hence, again by  \cref{comparison-sheaf}, $\mathscr{F}_{*}\mathcal{F}$ is a $T_2$-sheaf.   
\end{proof}

\begin{corollary}
\label{adjunction-sheaves}
Suppose $\mathscr{F}$ is continuous, cocontinuous, preserves and inverts coproducts. Then, the  adjunction of \cref{adjunction-presheaves} restricts to an adjunction between sheaves,
\begin{equation*}
\mathscr{F}^{*}:\Sh(\mathscr{C}_2,T_2) \leftrightarrows\ \Sh(\mathscr{C}_1,T_1): \mathscr{F}_{*}\text{.} 
\end{equation*}
\end{corollary}

Since we guaranteed already via the assumptions in  \cref{coproduct-inverting} and \cref{epsilon-iso} that the counit $\varepsilon$ of this adjunction is an isomorphism, the only difference to an adjoint equivalence is that the unit $\eta$ is an isomorphism.  

\begin{lemma}
The presheaf morphism
$\eta_{\mathcal{G}}$ is an isomorphism if $\mathcal{G}$ is a $T_2$-sheaf,  $\mathscr{F}$ has dense image
and $\mathscr{C}_2$ is extensive.
\end{lemma}

\begin{proof}
We recall that at an object $Y\in \mathscr{C}_2$ the map $\eta_{\mathcal{G}}|_{Y}:\mathcal{G}(Y) \to \mathscr{F}_{*}(\mathscr{F}^{*}\mathcal{G})(Y)$  sends $s\in \mathcal{G}(Y)$ to the family $\{s_{X,\psi}\}_{(X,\psi)\in \mathscr{F}/Y}$ with $s_{X,\psi} := \psi^{*}s \in \mathcal{G}(\mathscr{F}(X))$. For injectivity, we assume $s,s'\in \mathcal{G}(Y)$ such that $\psi^{*}s=\psi^{*}s'$ for all $X\in \mathscr{C}_1$ and all $\psi: \mathscr{F}(X) \to Y$. Since $\mathscr{F}$ has dense image, there exist $U_i \in \mathscr{C}_1$ and morphisms $\psi_i: \mathscr{F}(U_i) \to Y$ such that
\begin{equation*}
\psi:\coprod_{i\in I} \mathscr{F}(U_i) \to Y
\end{equation*} 
is universal  $T_2$-locally split. By assumption, we have $\psi_i^{*}s=\psi_i^{*}s'$ for all $i\in I$. This means that
\begin{equation*}
\{\psi_i^{*}s\}_{i\in I} = \{\psi_i^{*}s'\}_{i\in I} \in \prod_{i\in I}\mathcal{G}(\mathscr{F}(U_i))=\mathcal{G}(\coprod_{i\in I} \mathscr{F}(U_i))\text{.}
\end{equation*}
We also find that $\psi^{*}s=\{\psi_i^{*}s\}_{i\in I}$, since $\pr_i(\psi^{*}s)=\iota_i^{*}\psi^{*}s=\psi_i^{*}s$. This shows that $\psi^{*}s=\psi^{*}s'$, and hence, since $\mathcal{G}$ is a $T_2$-sheaf and  $\psi$ is universal  $T_2$-locally split,  $s=s'$. 

For surjectivity, we consider an element $\{s_{X,\psi}\}_{(X,\psi)\in \mathscr{F}/Y}$ in $\mathscr{F}_{*}(\mathscr{F}^{*}\mathcal{G})(Y)$, where the indices are morphisms $\psi: \mathscr{F}(X) \to Y$, and $s_{X,\psi}\in \mathcal{G}(\mathscr{F}(X))$. Since $\mathscr{F}$ has dense image, there exist $U_i\in \mathscr{C}_1$ and morphisms $\psi_i: \mathscr{F}(U_i) \to Y$ such that
\begin{equation*}
\psi:\coprod_{i\in I} \mathscr{F}(U_i) \to Y
\end{equation*} 
is universal  $T_2$-locally split. We note that the fibre product
\begin{equation*}
\alxydim{}{\mathscr{F}(U_i) \times_Y \mathscr{F}(U_j) \ar[r]^-{\pr_2}\ar[d]_-{\pr_1} & \mathscr{F}(U_j) \ar[d]^{\psi_j} \\ \mathscr{F}(U_i) \ar[r]_{\psi_i} & Y}
\end{equation*}
exists for all $i,j\in I$ because $\psi$ is universal and $\mathscr{C}_2$ is extensive. 
Again, since $\mathscr{F}$ has dense image, there exist $V_k\in \mathscr{C}_1$, with $k\in I_{ij}$, and morphisms $\pi_k: \mathscr{F}(V_k) \to \mathscr{F}(U_i) \times_Y \mathscr{F}(U_j)$ such that
\begin{equation*}
\pi: \coprod_{k\in I_{ij}} \mathscr{F}(V_k) \to \mathscr{F}(U_i) \times_Y \mathscr{F}(U_j)
\end{equation*}
is universal $T_2$-locally split. Since $\mathscr{F}$ is full and faithful, there exist unique morphisms $g_{k,1}: V_k \to U_i$ and $g_{k,2}:V_k \to U_j$ such that $\mathscr{F}(g_{k,a})= \pr_a \circ \pi_k$ for $a=1,2$, i.e., the diagram
\begin{equation*}
\alxydim{}{\mathscr{F}(V_k) \ar@/^1pc/[drr]^{\mathscr{F}(g_{k,2})} \ar@/_1pc/[ddr]_{\mathscr{F}(g_{k,1})} \ar[dr]_{\pi_k}\\&\mathscr{F}(U_i) \times_Y \mathscr{F}(U_j) \ar[r]^-{\pr_2}\ar[d]_-{\pr_1} & \mathscr{F}(U_j) \ar[d]^{\psi_j} \\ &\mathscr{F}(U_i) \ar[r]_{\psi_i} & Y}
\end{equation*}
is commutative.
We define $\psi_{k,a} : \mathscr{F}(V_k) \to Y$ by $\psi_{k,1}:= \psi_i \circ \pr_1 \circ \pi_k$ and $\psi_{k,2} := \psi_j \circ \pr_2 \circ \pi_k$,  and note that $\tilde\psi_k :=\psi_{k,1}=\psi_{k,2}$. Thus, we have $\mathscr{F}(g_{k,1})^{*}s_{U_i,\psi_i} =s_{V_k,\tilde \psi_k}=\mathscr{F}(g_{k,2})^{*}s_{U_j,\psi_j}$ and hence 
\begin{equation*}
\pi_k^{*}\pr_1^{*}s_{U_i,\psi_i} = \mathscr{F}(g_{1,k})^{*}s_{U_i,\psi_i}=\mathscr{F}(g_{2,k})^{*}s_{U_j,\psi_j}=\pi_k^{*}\pr_2^{*}s_{U_j,\psi_j}\text{.}
\end{equation*} 
Thus, we have
\begin{equation*}
(\pi_k^{*}\pr_1^{*}s_{U_i,\psi_i} )_{k\in I_{ij}} =(\pi_k^{*}\pr_2^{*}s_{U_j,\psi_j})_{k\in I_{ij}} \in \prod_{k\in I_{ij}}\mathcal{G}(\mathscr{F}(V_k))\text{,}
\end{equation*}
and hence,
\begin{equation*}
\pi^{*}\pr_1^{*}s_{U_i,\psi_i} =\pi^{*}\pr_2^{*}s_{U_j,\psi_j} \in \mathcal{G}(\coprod_{k\in I_{ij}} \mathscr{F}(V_k))\text{.} 
\end{equation*}
As $\pi$ is universal $T_2$-locally split and $\mathcal{G}$ is a $T_2$-sheaf, this implies $\pr_1^{*}s_{U_i,\psi_i}=\pr_2^{*}s_{U_j,\psi_j}$, for all $i,j\in I$. Thus, we have
\begin{equation*}
(\pr_1^{*}s_{U_i,\psi_i})_{i,j\in I} = (\pr_2^{*}s_{U_j,\psi_j})_{i,j\in I} \in \prod_{i,j\in I} \mathcal{G}(\mathscr{F}(U_i) \times_Y \mathscr{F}(U_j))\text{.}
\end{equation*}
Since $\mathcal{G}$ is extensive, the map
\begin{equation*}
\mathcal{G}( \coprod_{i\in I} \mathscr{F}(U_i) \times_Y \coprod_{i\in I} \mathscr{F}(U_i)) \to \prod_{i,j\in I} \mathcal{G}(\mathscr{F}(U_i) \times_Y \mathscr{F}(U_j)): s \mapsto ((\iota_i \times \iota_j)^{*}s)_{i,j\in I} 
\end{equation*}
is a bijection.
We consider the given element $(s_{U_i,\psi_i})_{i\in I} \in \prod_{i\in I} \mathcal{G}(\mathfrak{\mathscr{F}}(U_i))$ and its unique preimage $\tilde s \in \mathcal{G}(\coprod_{i\in I} \mathscr{F}(U_i))$, i.e., $\iota_i^{*}\tilde s = s_{U_i,\psi_i}$. We note that $\pr_1^{*}s_{U_i,\psi_i}=(\iota_i \times \iota_j)^{*}\pr_1^{*}\tilde s$. 
Then, we obtain that
\begin{equation*}
\pr_1^{*}\tilde s= \pr_2^{*}\tilde s \in \mathcal{G}( \coprod_{i\in I} \mathscr{F}(U_i) \times_Y \coprod_{i\in I} \mathscr{F}(U_i))
\end{equation*}
Since, again, $\mathcal{G}$ is a sheaf and $\psi$ is universal $T_2$-locally split, there exists a unique $s\in \mathcal{G}(Y)$ such that $\tilde s=\psi^{*}s$. We claim that $s$ is a preimage of the given element $(s_{X,\psi})_{(X,\psi)\in \mathscr{F}/Y}$.

To see this claim let $(X_0,\psi_0)\in  \mathscr{F}/Y$ be arbitrary; we have to show that $s_{X_0,\psi_0}=\psi_0^{*}s$. Similarly to the procedure above, we consider the pullback diagram
\begin{equation*}
\alxydim{}{\mathscr{F}(X_0) \times_Y \mathscr{F}(U_i) \ar[r]^-{\pr_2}\ar[d]_-{\pr_1} &  \mathscr{F}(U_i) \ar[d]^{\psi_i} \\ \mathscr{F}(X_0) \ar[r]_{\psi_0} & Y\text{.}}
\end{equation*}
Since $\mathscr{F}$ has dense image, there exists, for each $i\in I$, a set $K_i$, objects $V_k\in \mathscr{C}_1$ and morphisms $\pi_{k}:\mathscr{F}(V_k) \to \mathscr{F}(X_0) \times_Y \mathscr{F}(U_i)$ such that
\begin{equation*}
\pi: \coprod_{k\in K_i} \mathscr{F}(V_k) \to \mathscr{F}(X_0) \times_Y \mathscr{F}(U_i)
\end{equation*}
is universal $T$-locally split. Since $\mathscr{F}$ is fully faithful, there exists a unique $g_{0,k}:V_k \to U_i$ such that $\mathscr{F}(g_{0,k})= \pr_2 \circ \pi_k$, and also a unique $g_k:V_k \to X_0$ such that $\mathscr{F}(g_k)=\pr_1 \circ \pi_k$.  
Then,
\begin{align*}
\pi_k^{*}\pr_1^{*}s_{X_0,\psi_0} =\mathscr{F}(g_k)^{*}s_{X_0,\psi_0} &=s_{V_k,\psi_0 \circ \mathscr{F}(g_k)}
=s_{V_k,\psi_i \circ \mathscr{F}(g_{0,k})}=\mathscr{F}(g_{0,k})^{*}s_{U_i,\psi_i}\\&=\pi_k^{*}\pr_2^{*}s_{U_i,\psi_i}=\pi_k^{*}\pr_2^{*}\iota_i^{*}\tilde s=\pi_k^{*}\pr_2^{*}\iota_i^{*}\psi^{*}s=\pi_k^{*}\pr_2^{*}\psi_i^{*}s=\pi_k^{*}\pr_1^{*}\psi_0^{*}s\text{.}
\end{align*}
Thus, we have
$(\pi_k^{*}\pr_1^{*}s_{X_0,\psi_0} )_{k\in K_i} =(\pi_k^{*}\pr_1^{*}\psi_0^{*}s)_{k\in K_i}$ and hence, since $\mathcal{G}$ is extensive,
\begin{equation*}
\pi^{*}\pr_1^{*}s_{X_0,\psi_0} =\pi^{*}\pr_1^{*}\psi_0^{*}s \in \mathcal{G}(\coprod_{k\in K_{i}} \mathscr{F}(V_k)) 
\end{equation*}
As $\pi$  and $\pr_1$ are universal $T_2$-locally split and $\mathcal{G}$ is a $T_2$-sheaf, we obtain $s_{X_0,\psi_0}=\psi_0^{*}s$. 
\end{proof}

We obtain the following main result, which has appeared in numerous versions in many references, e.g.
\cite[App. 4]{MacLane1992}, \cite[C2.2]{Johnstone2002}, \cite[\S 5]{Caramello2019}, \cite[Lem. 44]{metzler}.

\begin{theorem}
\label{comparison-lemma}
Suppose $\mathscr{F}:(\mathscr{C}_1,T_1) \to (\mathscr{C}_2,T_2)$ is continuous, cocontinuous, has dense image, and preserves and inverts coproducts, and suppose $\mathscr{C}_2$ is extensive. Then, the functors
\begin{equation*}
\mathscr{F}^{*}: \Sh(\mathscr{C}_2,T_2) \rightleftharpoons \Sh(\mathscr{C}_1,T_1): \mathscr{F}_{*}
\end{equation*}
establish an adjoint equivalence. 
\end{theorem}

\section{Examples of sites}

\label{examples-of-sites}

\subsection{Smooth manifolds}

\label{smooth-manifolds}

The following lists some well-known facts about the category $\Man$ of smooth manifolds (finite-dimensional, Hausdorff, and second countable).
\begin{itemize}

\item
$\Man$ is neither complete nor cocomplete.

\item 
$\Man$ is extensive, with initial object the empty manifold.

\item
A smooth map is:
\begin{itemize}
\item 
universal if and only if it is a submersion.

\item
an epimorphism if and only if
it has a dense image.

\item
an effective epimorphism if and only if it is a submersion with dense image. 

\item
a universally effective epimorphism if and only if it is a surjective submersion. 
\end{itemize}

\item
The discrete Grothendieck topology consists of all
submersions.
\item
The canonical Grothendieck topology $T_{\Man}$ consists of all
surjective submersions. 
\end{itemize} 

\begin{example}
\label{GTs-on-Man}
We consider the following Grothendieck topologies on the category $\Man$:
\begin{itemize}

\item
$T_{op}$ -- the coverings are families  $(\pi_i: U_i \incl M)_{i\in I}$ of  embeddings, such that $(\pi_i(U_i))_{i\in I}$ is a cover of $M$ by open sets. 

\item
$T_{cop}$ -- the coverings are as above, but with $I$ countable.

\item
$T_{sld}$ -- the coverings are surjective local diffeomorphisms.

\item
$T_{susu}$ -- the coverings are surjective submersions; $T_{susu}=T_{\Man}$.

\end{itemize}
In all cases, it is straightforward to verify the axioms of a Grothendieck topology. 
\end{example}

\begin{lemma}
\label{topologies-on-man}
All Grothendieck topologies in \cref{GTs-on-Man} are all equivalent; they are all subcanonical and local. 
\end{lemma}

\begin{proof}
We have inclusions $T_{cop}\subset T_{op}$ and $T_{sld}\subset T_{susu}$. The first one is an equivalence since every open cover of a smooth manifold has a countable refinement (\cref{equivalence-by-refinement}). Moreover, since the category $\Man$ admits countable disjoint unions (of manifolds with the same dimension),   $T_{cop}$ is singletonizable, and $T_{cop} \sim T_{cop}^{sing}$ by \cref{singletonization}.
Again, obviously, $T_{cop}^{sing}\subset T_{sld}$.
Thus, we have
\begin{equation*}
T_{op} \sim T_{cop} \prec T_{sld} \prec T_{susu}\text{.}
\end{equation*}
Since every surjective submersion is $T_{op}$-locally split, \cref{checking-equivalence} implies $T_{susu}\prec T_{op}$. This shows that all four Grothendieck topologies are equivalent. Since $T_{susu}=T_{\Man}$, $T_{susu}$ is subcanonical, and also singleton; hence, all four Grothendieck topologies are subcanonical. Locality is proved in \cite[Prop. 9.40]{Meyer2014}.
\end{proof}

\begin{remark}
Due to \cref{topologies-on-man}, all Grothendieck topologies in \cref{GTs-on-Man}  induce the same universal locally split maps, namely the class of surjective submersions $\Uni(T)=T_{susu}=T_{\Man}$, for all $T$ from  \cref{GTs-on-Man}  (cf. \cite[Example 3.9]{Roberts2012}).
\end{remark}

The following list shows that the geometric structures discussed in \cref{geometry-on-sites} and the sheaves discussed in \cref{sheaves}, specialize to the standard structures considered for the category of smooth manifolds.
\begin{enumerate}

\item 
$\Man$-categories (\cref{internal-category}) and $\Man$-groupoids are precisely the usual Lie categories and Lie groupoids: since source and target map are required to be universal, hence are  submersions, and are surjective anyway. Weak equivalences (\cref{weak-equivalence}) are the usual weak equivalences between Lie groupoids (see, e.g. \cite{lerman1,metzler,Hoyo2013}), and coincide with the weak equivalences in \cite[cf. Thm. 3.28]{Meyer2014} (because the Grothendieck topologies on $\Man$ are local).  

\item
The usual definition of a smooth principal bundle on a smooth manifold gives our notion of a $T_{op}$-local  principal bundle. By \cref{topologies-on-man,principal-bundles-equivalence} these are the same as the local principal bundles in any of the other  Grothendieck topologies of \cref{GTs-on-Man}. 

\item
Locally trivial bibundles in $\Man$ are precisely the smooth bibundles, as e.g. discussed in \cite{lerman1,metzler,Blohmann2008,Hoyo2013,Nikolaus}. The statements summarized in \cref{bicat-bibundles} were proved \quot{manually} in these sources, but now simply follow from the results of Meyer-Zhu.  

\item
For sheaves on $\Man$ one typically uses the traditional sheaf condition of \cref{traditional-sheaf} for the Grothendieck topology $T_{op}$. Since $T_{op}$ is superextensive, every traditional sheaf is a $T_{op}$-sheaf in our sense, by \cref{comparison-sheaf}. Conversely, every $T_{op}$-sheaf is by \cref{topologies-on-man}   also a $T_{cop}$-sheaf, and since $T_{cop}$ is singletonizable, it is also a traditional sheaf for $T_{cop}$. But traditional sheaves w.r.t. open covers are the same as traditional sheaves for countable open covers, since every open cover is \quot{tautologically equivalent} to a countable refinement (\cite[Def. 7.8.2 (2) \& Lem. 7.8.4]{stacks-project}).
Thus, the traditional sheaves on $\Man$ are precisely our sheaves, for any of the four equivalent Grothendieck topologies  of \cref{GTs-on-Man}.
\end{enumerate}

\begin{warnung}
\label{traditional-sheaf-bad}
The \emph{traditional sheaf condition} leads \emph{not} to the same category of sheaves for all four Grothendieck topologies   of \cref{GTs-on-Man}.
 Indeed, a traditional $T_{susu}$-sheaf is not necessarily a traditional $T_{op}$-sheaf. For example, let $\mathcal{F}:\Man^{op} \to \Set$ be the constant presheaf that assigns to each manifold $M$ the set $\Z$ (or any other non-singleton set), and to every morphism the identity $\id_\Z$. We claim that this is a traditional $T_{susu}$-sheaf: if $\pi:Y \to M$ is a surjective submersion, then 
\begin{equation*}
\alxydim{}{\Z \ar[r] & \Z \arr[r] & \Z }
\end{equation*}  
with all maps identities, is an equalizer diagram. But  $\mathcal{F}$ is a not traditional $T_{op}$-sheaf, for if it was,  since $T_{op}$ is superextensive, \cref{traditional-sheaf-extensive} implies that $\mathcal{F}$ is extensive, in particular, $\mathcal{F}(\emptyset)=*$, a contradiction. 
This conflict is resolved by our modified sheaf condition that includes extensiveness: $\mathcal{F}$ is not extensive and hence neither a $T_{susu}$-sheaf nor a $T_{op}$-sheaf. 
\end{warnung}

\subsection{Opens and submanifolds}

We will first consider   \emph{embedded submanifolds} $X \subset \R^{n}$, for all possible $n\in \N$. These form a small full subcategory $\Subman\subset \Man$, which is -- via Whitehead's embedding theorem -- equivalent to $\Man$. It is hence extensive, and a smooth map is universal / an epimorphism / an effective epimorphism / a universally effective epimorphism in $\Subman$ if and only if it is so in $\Man$. We may consider on $\Subman$ the Grothendieck topology $T_{susu}|_{\Subman}$, whose coverings are all surjective submersions between embedded submanifolds. 
We denote this again by $T_{susu}$. Similarly, one can restrict any of the other equivalent Grothendieck topologies of \cref{GTs-on-Man} to $\Subman$ to obtain equivalent Grothendieck topologies there.    

Next we restrict to \emph{open submanifolds}, i.e. open subsets $U \subset \R^{n}$, yielding a full subcategory $\Open \subset \Subman$, which is, however, not equivalent to $\Subman$. $\Open$ is extensive, with $\emptyset \subset \R^{0}$ as the initial object. Smooth maps in $\Open$ are (effective) epimorphisms if they are so in $\Subman$. It is natural to consider the Grothendieck topology $T_{op}|_{\Open}$ of  all open covers, which we denote again by $T_{op}$.
It is also possible to restrict $T_{susu}$ to $\Open$, but one has to be careful since a general surjective submersion in $\Open$ is not necessarily universal in $\Open$; to see this, consider the  diagram
\begin{equation*}
\alxydim{}{ & \R^2 \setminus \{0\} \ar[d]^{\|..\|} \\ \R^{0} \ar[r]_-{const_1} & \R \setminus \{0\}}
\end{equation*}
in $\Open$, whose vertical map is a surjective submersion, but whose limit (i.e., $S^1$) only exists in $\Subman$ but not in $\Open$.  
Hence, the coverings of $T_{susu}|_{\Open}$ are only \emph{some} surjective submersions, and we currently do not have an explicit answer. Yet, we can prove that $T_{susu}|_{\Open}\sim T_{op}$. Indeed, suppose a smooth map $\pi:Y \to U$ in $\Open$ is $T_{op}$-locally split. We choose a countable open cover $(U_n)_{n\in \N}$ of $U$ with local sections $\rho_i: U_i \to Y$, and consider 
\begin{equation}
\label{dsfsdfsdf}
\tilde Y:=\bigcup_{n\in \N} \{(x,r)\sep x\in U_n\text{ and } n-0.1<r<n+0.1\}\subset \R^{k+1}\text{,}
\end{equation}
a  union of countably many open sets in $\R^{k+1}$;  hence, an object in $\Open$. We have a smooth map $\rho:\tilde Y \to Y$ sending $(x,r)$ from the $n\in \N$ component to $\rho_n(x)$. Moreover, we consider the projection $\tilde\pi:\tilde Y \to U$ given by $(x,r) \mapsto x$, which is a surjective submersion, and satisfies $\pi \circ \rho=\tilde\pi$.
We claim that $\tilde\pi:\tilde Y \to U$ is universal in $\Open$; this shows that $\tilde\pi$ is a $T_{susu}|_{\Open}$-covering and that $\pi$ is $T_{susu}|_{\Open}$-locally split, thus, $T_{op}\prec T_{susu}|_{\Open}$. 
To show the claim, let $f:W \to U$ be a morphism in $\Open$. One can check that the fibre product $W \times_U \tilde Y$ is given by the object 
\begin{equation*}
\tilde W := \bigcup_{n\in \N} \{(y,r) \sep y\in f^{-1}(U_n)\text{ and }n-0.1 < r < n+0.1 \}
\end{equation*}  
in $\Open$
together with the maps $\tilde W \to W,(y,r)\mapsto y$ and $\tilde W \to \tilde Y,(y,r)\mapsto (f(y),r)$.
Conversely, since any surjective submersion splits over an open cover, we have $T_{susu}|_{\Open}\prec T_{op}$, and hence $T_{op}\sim T_{susu}|_{\Open}$. One can show like in \cref{smooth-manifolds} that traditional sheaves on $\Open$ (w.r.t. to open covers) are precisely our $T_{op}$-sheaves on $\Open$.

\subsection{Presheaves and sheaves}

\label{presheaves-and-sheaves}

Let $\mathscr{C}$ be a category. We collect some well-known facts about the category $\PSh(\mathscr{C})$ of presheaves on $\mathscr{C}$:
\begin{itemize}
\item

it is complete and cocomplete; limits and colimits computed object-wise.
In particular, every morphism is universal. 
\item
it is extensive, with  initial object the empty presheaf (which assigns the empty set to every object, and the empty map to every morphism).  

\item
for a morphism of presheaves, the following properties are equivalent:  
\begin{itemize}

\item 
it is an epimorphism 

\item
it is an effective epimorphism

\item
it is a universally effective epimorphism

\item
it is object-wise surjective.

\end{itemize}

\item
the canonical Grothendieck topology $T_{\PSh(\mathscr{C})}$ consists of all epimorphisms.

\end{itemize}
Let $T$ be a Grothendieck topology on $\mathscr{C}$. 
We induce on $\PSh(\mathscr{C})$ the   singleton Grothendieck topology $\Pre(T)$, in which a morphism $\phi: \mathcal{F} \to \mathcal{G}$ of presheaves is a covering if and only if  for every object $X\in \mathscr{C}$ and every element $\psi\in \mathcal{G}(X)$ there exists a covering $(\pi_i:U_i \to X)_{i\in I}$ and elements $\psi_i'\in \mathcal{F}(U_i)$ such that $\pi_i^{*}\psi=\phi_{U_i}(\psi_i')$. 
It is straightforward to see that this indeed defines a singleton Grothendieck topology on $\PSh(\mathscr{C})$.
We remark that $\Pre(T)$ is in general not subcanonical, since its coverings are not necessarily (effective) epimorphisms.  

Next we consider the full subcategory $\Sh(\mathscr{C},T)\subset \PSh(\mathscr{C})$ of $T$-sheaves on $\mathscr{C}$. For simplicity, we assume that $T$ is superextensive and singletonizable; then, by \cref{comparison-sheaf}, this is the ordinary category of sheaves. It is hence complete, cocomplete, and extensive, and the inclusion functor $\Sh(\mathscr{C},T) \to \PSh(\mathscr{C})$ preserves limits.  It is well-known that
\begin{equation}
\label{canonical-topology-on-Sh}
\Pre(T)|_{\Sh(\mathscr{C},T)} = T_{\Sh(\mathscr{C},T)}\text{.}
\end{equation}
Indeed,  for a morphism $\phi: \mathcal{F} \to \mathcal{G}$ between $T$-sheaves the following are equivalent, see, e.g. \cite[Lem. 7.11.2 \& 7.11.3]{stacks-project}: 
\begin{enumerate}

\item 
it is an  epimorphism in $\Sh(\mathscr{C},T)$ 

\item
it is an effective epimorphism in $\Sh(\mathscr{C},T)$

\item
it is a $\Pre(T)$-covering. 

\end{enumerate}
Since epimorphisms are stable under pullback in $\Sh(\mathscr{C},T)$, above conditions are also equivalent to: 
\begin{enumerate}

\item[4.]
it is a universally effective epimorphism in $\Sh(\mathscr{C},T)$.

\end{enumerate}
This proves \cref{canonical-topology-on-Sh}.

\subsection{Diffeological spaces}

\label{diffeological-spaces}

A comprehensive treatment of diffeological spaces from an analytical point of view is given in the book of Iglesias-Zemmour \cite{iglesias1}. More category-oriented treatments are \cite{baez6,Wu2012,Kihara2018}.
We recall that a \emph{diffeological space} is a set $X$ together with a collection of maps $c:U \to X$ called \emph{plots}, where $U \in \Open$, such that:
\begin{enumerate}[(1)]

\item
every constant map is a plot. 

\item
if $f:U \to V$ is a morphism in $\Open$ and $c: V \to X$ is a plot, then $c \circ f$ is a plot.

\item
if $(U_i)_{i\in I}$ is an open cover of $U \in \Open$, and $c_i:U_i \to X$ are plots such that $c_i|_{U_i \cap U_j} = c_j|_{U_i \cap U_j}$ for all $i,j\in I$, then there exists a unique plot $c:U \to X$ such that $c_i=c|_{U_i}$ for all $i\in I$.  

\end{enumerate} 
A \emph{smooth map} between diffeological spaces is a map $g: X \to Y$ that sends plots $c:Y \to X$ to plots $c \circ g$. The category of diffeological spaces is denoted by $\Diff$.

\begin{remark}
If $X$ is a diffeological space, we may consider the functor
\begin{equation*}
\mathcal{D}_X: \Open^{op} \to \Set 
\end{equation*}
that sends $U \in \Open$ to the set of plots with domain $U$, and a morphism $f:U \to V$ to the map $c \mapsto c \circ f$, which is well-defined due to (2). Due to (3), this is a traditional sheaf for $T_{op}$. Since $T_{op}$ is superextensive, $\mathcal{D}_X$ is a $T_{op}$-sheaf. This defines a fully faithful functor
\begin{equation*}
\Diff \to \Sh(\Open,T_{op})\text{.}
\end{equation*}   
\end{remark}

An important class of smooth maps between diffeological spaces are so-called subductions. 

\begin{definition}
A smooth map $\pi:Y \to X$ is called \emph{subduction} if for all plots $c:U \to X$ and all $x\in U$ there exists an open neighborhood $x\in U_x \subset U$ and a smooth map $\sigma:U_x \to Y$ such that $\pi \circ \sigma=c|_{U_x}$. 
\end{definition}

The following is a straightforward observation.
We call a smooth map
$\pi:Y \to X$  \emph{plot-wise local} if its pullback $c^{*}\pi: c^{*}Y \to U$ along every plot $c:U \to X$ is $T_{op}$-locally split. Then, we have: 
\begin{lemma}
A smooth map is plot-wise local
if and only if it is a subduction.
\end{lemma}

Subductions are important due to the following result. 

\begin{proposition}
The category $\Diff$ is complete, cocomplete, and extensive. A smooth map $g:X \to Y$ is an epimorphism if and only if it is surjective. Moreover,  the following statements are equivalent:

\begin{enumerate}[(i)]

\item
$f$ is an effective epimorphism. 

\item
$f$ is a universally effective epimorphism.

\item
$f$ is a subduction. 

\end{enumerate}
In particular, the canonical Grothendieck topology $T_{\Diff}$ consists of all subductions. 
\end{proposition}

\begin{proof}
Completeness and cocompleteness, together with explicit descriptions of all kinds of limits and colimits, and be found in \cite[\S 3]{baez6}.  
The statement about  epimorphisms is  \cite[Prop. 31]{baez6}.

We show the equivalence of (i) with (iii).  Suppose first that $\pi:Y \to X$ is a subduction. We consider a cocone for the diagram $Y^{[2]} \rightrightarrows Y$, i.e., a smooth map $f:Y \to X'$ such that $f(y_1)=f(y_2)$ for all $(y_1,y_2)\in Y^{[2]}$. Since $\pi$ is surjective, there exists a unique map $u: X \to X'$ such that $f= u \circ \pi$. It remains to prove that $u$ is smooth. This can be checked locally on the plots. If $c:U \to X$ is a plot, then, because $\pi$ is a subduction, each point $x\in U$ has an open neighborhood $x\in U_x \to U$ together with a smooth section $\sigma: U_x \to Y$. Then, $(u \circ c)|_{U_x}=u\circ \pi \circ \sigma=f \circ \sigma$, which is smooth. This shows that $u$ is smooth and hence, that $\pi:Y \to X$ is the colimit.  

Now suppose $\pi:Y \to X$ is an effective epimorphism. We define on $Y$ the equivalence relation $y_1\sim y_2 \Leftrightarrow \pi(y_1)=\pi(y_2)$, and  $X' := Y/\sim$ be the quotient space, and $p:Y \to X'$ be the canonical projection. Projections to quotients are always subductions. Since $\pi$ is effective, there exists a smooth map $u:X \to X'$ such that $u \circ \pi = p$. Moreover, $u$ is a bijection, since the quotient computes the colimit in the category of sets. If now $c:U \to X$ is a plot, then $u \circ c$ is a plot of $X'$. Thus, every point $x\in X$ has an open neighborhood $x\in U_x \subset U$ with a lift $\sigma: U_x \to Y$, i.e., $p \circ \sigma = (u \circ c)|_{U_x}$. Hence, $u \circ \pi \circ \sigma =(u \circ c)|_{U_x}$, showing that $\pi \circ \sigma = c|_{U_x}$. Thus, $\pi$ is a subduction.

It is easy to see that the pullback of a subduction is a subduction \cite[Lemma 2.23]{Schaaf2020a}.
This shows that subductions are even universally effective epimorphisms. 
\end{proof}

By identifying the subductions as the coverings of the canonical Grothendieck topology on $\Diff$, we have our first explicit example of a Grothendieck topology, and mostly we will write $T_{sudu}$ instead of $T_{\Diff}$ in order to emphasize that we know exactly what the canonical coverings are. That $T_{sudu}$ is a Grothendieck topology is also shown in \cite[Thm. 10.1]{Watts2022}.
In the following we introduce some more examples.

\begin{definition}
A smooth map $\pi:Y \to X$ is called \emph{submersion} if  for all $y\in Y$, all plots $c:U \to X$ and all $x\in U$ with $c(x)=\pi(y)$, there exists an open neighborhood $x\in U_x \subset U$ and a smooth map $\sigma:U_x \to Y$ such that $\pi \circ \sigma=c|_{U_x}$ and $\sigma(x)=y$. 
\end{definition}

\begin{remark}
\label{Local-subductions-are-subductions}
Every  surjective submersion is a subduction. Submersions are called \quot{local subductions} in \cite{iglesias1}. 
\end{remark}

\begin{lemma}
\label{submersions-diffeo}
Surjective submersions form a subcanonical singleton Grothendieck topology $T_{susu}$ on $\Diff$.
\end{lemma}

\begin{proof}
The composition of submersions is a submersion \cite[2.17]{iglesias1}.
That the pullback of a submersion is again a submersion is easy to see, and that diffeomorphisms are submersions is clear. This shows that we have a Grothendieck topology.
Since $T_{susu} \subset T_{sudu}$, it is subcanonical. 
\end{proof}

A similar statement to \cref{submersions-diffeo} can be found in \cite[Prop. 10.5]{Watts2022}. 
The next round of examples of Grothendieck topologies uses so-called \emph{D-open sets}. If $X$ is a diffeological space, a subset $A \subset X$ is called \emph{D-open}, if its preimage under every plot $c:U \to X$ is open in $U$. D-open sets equip $X$ with a topology, the \emph{D-topology}. We fix the following definitions.

\begin{definition}
Let $X$ be a diffeological space.
\begin{enumerate}[(a)]

\item 
A \emph{D-open cover} of a diffeological space $X$ is a family $(\pi_i: X_i \to X)_{i\in I}$ of smooth maps with $D$-open images $A_i:=\pi_i(X_i) \subset X$, such that $\pi_i: X_i \to A_i$ are diffeomorphisms (w.r.t. the subspace diffeology on $A_i$) and $(A_i)_{i\in I}$ covers $X$. 

\item
A \emph{numerable D-open cover} is a D-open cover such that  the open cover $(A_i)_{i\in I}$ of the topological space $X$ is numerable. 

\item
A smooth map $\pi:Y \to X$ between diffeological spaces is called a \emph{D-local diffeomorphism} if  each point $y\in Y$ has a D-open neighborhood $A\subset Y$ such that $\pi(A)\subset X$ is D-open and  $\pi|_A : A \to \pi(A)$ is a diffeomorphism (w.r.t. the subspace diffeologies).  

\item
A smooth map $\pi:Y \to X$ is called \emph{ D-submersion} if   for every point $y\in Y$ there exists a D-open neighborhood $A\subset X$ of $\pi(x)$ together with a smooth map $\sigma:A \to Y$ such that $\pi \circ \sigma=\id_A$ and $\sigma(\pi(x))=y$.

\item
A smooth map $\pi:Y \to X$ is called \emph{D-locally split} if every point $x\in X$ has a D-open neighborhood $A \subset X$ together with a smooth map $\sigma:A \to Y$ such that $\pi \circ \sigma=\id_A$.

\end{enumerate}
\end{definition}

The next lemma is easy to check.

\begin{lemma}
We have the following Grothendieck topologies on $\Diff$:
\begin{itemize}

\item 
Numerable D-open covers covers form a  superextensive and singletonizable Grothendieck topology $T_{numDop}$. 

\item 
D-open covers covers form a  superextensive and singletonizable Grothendieck topology $T_{Dop}$. 

\item
Surjective D-local diffeomorphisms form a singleton Grothendieck topology $T_{suDld}$.

\item
Surjective D-submersions form a singleton Grothendieck topology $T_{suDsu}$.

\item
D-locally split maps form a singleton Grothendieck topology $T_{Dlsplit}$. 

\end{itemize}
Moreover, all Grothendieck topologies above are subcanonical. 
\end{lemma}

\begin{proposition}
\label{comparison-GTs-on-Diff}
The Grothendieck topologies collected above obey the following relations:
\begin{equation*}
\alxydim{@C=3em}{&&&& T_{Dlsplit}  \ar@{_(->}[dr]  \\T_{numDop} \ar[r] & T_{Dop} \ar[r]^{\sim}  & T_{suDld} \ar@{^(->}[r]^{\sim}  & T_{suDsu} \ar@{^(->}[ur]^-{\sim} \ar@{_(->}[dr] && T_{sudu} \\ &&&& T_{susu} \ar@{^(->}[ur] }
\end{equation*}
Here,  an arrow \quot{$\;\alxydim{@C=0.5em}{\ar@{^(->}[r] &}$} stands for \quot{$\,\subset$}, an arrow \quot{$\alxydim{@C=0.5em}{\ar[r] &}$} stands for \quot{$\prec$}, and a decoration with \quot{$\;\sim$} means that an equivalence is induced.
\end{proposition}

\begin{proof}
The inclusion relations follow from the following elementary statements: 
\begin{itemize}

\item 
Every D-local diffeomorphism is a D-submersion.

\item
Every D-submersion is a submersion.

\item
Every D-submersion is D-locally split. 

\item
Every D-locally split map $f:Y \to X$ is a subduction.

\item
Every surjective submersion is a subduction; this is \cref{Local-subductions-are-subductions}. 

\end{itemize}
The relation $T_{Dop}\prec T_{suDld}$ holds because $T_{Dop}\sim \Sing(T_{Dop})\subset T_{suDld}$. The equivalences come about because every D-locally split map has a refinement through a D-open cover. 
\end{proof}

\begin{example}
Differential forms on $\Diff$ form a sheaf for $T_{sudu}$; hence, by \cref{comparison-GTs-on-Diff}, also a sheaf for all other Grothendieck topologies studied here. 
\end{example}

The following list shows that the geometric structures discussed in \cref{geometry-on-sites} and the sheaves discussed in \cref{sheaves}, specialize to various versions considered for the category of diffeological spaces:
\begin{enumerate}[(a)]

\item
Diffeological groupoids as used, e.g., in \cite{iglesias1,Blohmann2021,Watts2022,Schaaf2020} are $\Diff$-groupoids in the sense of \cref{internal-category}. In particular, there are no extra conditions on source and target maps. However, as described in \cref{internal-categories}, source and target maps are always $T$-locally split, for any $T$. In particular, they are always subductions; this is \cite[Prop. 3.17]{Schaaf2020}.         
 $T_{sudu}$-weak equivalences between diffeological groupoids appear in \cite[Def. 5.1]{Schaaf2020} and \cite[Def. 4.3]{Watts2022}.

\item
Diffeological principal $G$-bundles (for diffeological groupoids $G$)  have been introduced in \cite[Def. 4.19]{Schaaf2020}, matching precisely our general definition of $T_{sudu}$-locally trivial principal $G$-bundles (\cref{principal-G-bundle,locally-trivial-bundle}). If $G$ is a diffeological \emph{group}, gives back immediately earlier definitions in \cite{waldorf9,Minichiello2022}. In \cite[Def. 5.1]{Krepski2021}, the projection $p$ of a principal $G$-bundle $p:P \to X$ is required to be the quotient map of the $G$-action on $P$, and quotient maps are subductions. Conversely, if the projection is a subduction, the fact that subductions are effective epimorphisms and that the shear map is a diffeomorphism, shows that
\begin{equation*}
\alxydim{}{P \times G \arr[r] & P \ar[r]^{p} & X} 
\end{equation*}  
is a coequalizer, hence that $p$ is a quotient map. Thus, the principal $G$-bundles of \cite{Krepski2021} are the same as ours. Finally, the principal $G$-bundles of \cite{Krepski2021} coincide with the ones of \cite[8.11]{iglesias1}; this is \cite[Rem. 5.2]{Krepski2021}.
In \cite{Magnot2017} $T_{numDop}$-locally trivial principal $G$-bundles are studied; they are, in particular, $T_{sudu}$-locally trivial.

\item
Locally trivial diffeological bibundles for $T_{sudu}$ are in use in   \cite{Schaaf2020,Schaaf2020a,Watts2022}. These reference all construct the same  bicategory of diffeological groupoids,  which coincides precisely with the bicategory $(\Diff,T_{sudo})$-$\Grpd^{bi}$ from \cref{bicat-bibundles:a}, for instance \cite[Thm. 8.12]{Watts2022} and \cite[Thm. 4.51]{Schaaf2020}. They also derive some conclusions that we also get from the general theory of Meyer-Zhu, for instance, \cite[Lemma 8.15]{Watts2022} is \cref{bicat-bibundles:c}, and \cite[Thm. 8.19]{Watts2022} follows from \cite[Thm. 7.15 \& 3.23]{Meyer2014}.

\end{enumerate}

\subsection{Topological spaces}

The following lists some well-known facts about the category $\Top$ of topological spaces:
\begin{itemize}

\item
$\Top$ is complete and cocomplete; in particular, every continuous map is universal.

\item 
$\Top$ is extensive, with initial object the empty space.

\item
A continuous map $f:X \to Y$ is
\begin{itemize}

\item
an epimorphism if and only if
it is surjective. 
\item
an effective epimorphism if and only if it is quotient map. 

\item
a universally effective epimorphism if and only if it is  surjective, and for every point $y\in Y$ and every covering $(U_i)_{i\in I}$ of $f^{-1}(\{y\})$ by open subsets $U_i$, there exist finitely many open subsets $U_1,...,U_n$ such that $f(U_1)\cup ... \cup f(U_n)$ is a neighborhood of $y$; see  \cite{Day1970}. 
\end{itemize}

\item
The discrete Grothendieck topology consists of all
continuous maps.\item
The canonical Grothendieck topology $T_{\Top}$ consists, as always, of all universally effective epimorphisms. 
\end{itemize}

\begin{example}
\label{GTs-on-Top}
We consider the following Grothendieck topologies on the category of topological spaces:
\begin{enumerate}[leftmargin=4em,itemindent=0pt,labelwidth=*]

\item 
[$T_{op}$ :]
the coverings are families  $(\pi_i: U_i \incl X)_{i\in I}$ of  topological embeddings, such that $(\pi_i(U_i))_{i\in I}$ is an open cover of $X$. 

\item
[$T_{nop}$ :] 
as before, but such that the open cover $(\pi_i(U_i))_{i\in I}$ is numerable, i.e., it admits a subordinate partition of unity. 

\item
[$T_{slh}$ :]
the coverings are surjective local homeomorphisms.

\item
[$T_{sutsu}$ :]
the coverings are surjective topological submersions, i.e., surjective continuous maps $f:X \to Y$ such that for each $x\in X$ there is an open neighborhood $U\subset Y$ of $f(x)$ with a section $\sigma: U \to X$ satisfying $\sigma(f(x))=x$. 

\item
[$T_{sop}$ :]
the coverings are surjective open maps; this is used in \cite{pronk} and in most work about groupoid $C^{*}$-algebras; see \cite[§9.2.6]{Meyer2014}. 

\item
[$T_{lsplit}$ :]
the coverings are locally split maps, i.e. continuous maps $f:X \to Y$ such that each $y\in Y$ has an open neighborhood $U \subset Y$ with a section $\sigma:U \to X$.

\item
[$T_{sur}$ :]
the coverings are surjective maps. 

\end{enumerate}
We refer to \cite[§9.2]{Meyer2014}  for an even longer list of singleton Grothendieck topologies on $\Top$, together with explicit proofs of the axioms.
\end{example}

\begin{lemma}
\label{Comparison-of-GTs-on-Top}
The Grothendieck topologies of \cref{GTs-on-Top} obey the following relations:
\begin{equation*}
\alxydim{@C=3em}{&&&& T_{lsplit}  \ar@{_(->}[dr]  \\T_{nop} \ar@{^(->}[r]  &T_{op} \ar[r]^{\sim}  & T_{slh} \ar@{^(->}[r]^{\sim}  & T_{sutsu} \ar@{^(->}[ur]^-{\sim} \ar@{_(->}[dr] && T_{sur} \\ &&&& T_{sop} \ar@{^(->}[ur] }
\end{equation*}
Here,  an arrow \quot{$\;\alxydim{@C=0.5em}{\ar@{^(->}[r] &}$} stands for \quot{$\,\subset$}, an arrow \quot{$\alxydim{@C=0.5em}{\ar[r] &}$} stands for \quot{$\prec$}, and a decoration with \quot{$\;\sim$} means that an equivalence is induced.
Moreover, all Grothendieck topologies are either singleton or singletonizable, and $T_{op}$ is superextensive. Finally, all Grothendieck topologies except $T_{sur}$ are subcanonical and local.
\end{lemma}

\begin{proof}
The inclusion arrows are mostly clear; maybe the only non-trivial but still elementary statement is that surjective topological submersions are open maps.
Since $T_{op}$ is singletonizable, we have $T_{op}\sim \Sing(T_{op}) \subset T_{slh}$; this shows the only proper \quot{$\prec$}-relation. Since every locally split map splits over an open cover, we have $T_{lsplit}\prec T_{op}$; this proves the three equivalences at once. That the singleton Grothendieck topologies except $T_{sur}$ are subcanonical can be proved by checking that their coverings are universally effective epimorphisms. Since $T_{op}\sim\Sing(T_{op})\subset T_{\Top}$, we get $\Uni(T_{op})\sim \Uni(\Sing(T_{op}))\subset \Uni(T_{\Top})\sim T_{\Top}$; this shows that $T_{op}$ is subcanonical. Since $T_{nop}\subset T_{op}$ implies $\Uni(T_{nop})\subset \Uni(T_{op})$, we see that $T_{nop}$ is also subcanonical. For locality, we have $\Uni(T)=T_{lsplit}$ for $T\in\{ T_{op},T_{slh},T_{sutsu},T_{lsplit}\}$ for which the condition is easy to check.  $\Uni(T_{nop})$ is the Grothendieck topology of \quot{maps with local continuous sections and partitions of unity}, which is proved in \cite[Prop. 9.24]{Meyer2014} to satisfy their Assumption 2.6, equivalent to our locality.  Finally, for $T_{sop}=\Uni(T_{sop})$, and one can check explicitly that it is local. 
\end{proof}

\noindent
The following list shows that the geometric structures discussed in \cref{geometry-on-sites} and the sheaves discussed in \cref{sheaves} specialize to the standard structures considered for the category of topological  spaces.
\begin{enumerate}

\item 
$\Top$-categories (\cref{internal-category}) and $\Top$-groupoids are precisely the usual topological categories and topological groupoids, where no extra condition is put for source and target maps.  Nonetheless, as described in \cref{internal-categories}, source and target maps are automatically $T$-locally split, for any $T$. $T$-weak equivalences (for $T\in \{T_{op}, T_{slh}, T_{sutsu}, T_{lsplit} \}$) between topological groupoids (\cref{weak-equivalence}) are called \quot{essential equivalences} in \cite[Def. 58]{metzler}. 
Carchedi considers $T$-weak equivalences between topological groupoids for arbitrary $T$ \cite[Def. I.2.20]{Carchedi2011}, matching exactly our general  \cref{weak-equivalence}.
Freed-Hopkins-Teleman \cite{Freed2011a} use a non-standard notion called \quot{local equivalence}, with another meaning of fully faithfulness, but our (standard) notion of $T$-essential surjectivity.

\item
The usual definition of a continuous principal bundle on a topological space gives our notion of a $T_{op}$-local  principal bundle. By \cref{Comparison-of-GTs-on-Top,principal-bundles-equivalence} these are the same $T$-local principal bundles, for $T \in \{T_{sutsu},T_{lsec}\}$. $T_{nop}$-local principal bundles are more special; these are the ones that are classified by Milnor's classifying space. $T_{sop}$-local principal bundles are considered in \cite[Def. 61]{metzler}, these are  more general.
Carchedi discusses general $T$-locally trivial principal bundles for topological groupoids \cite[I.2.13]{Carchedi2011}.

\item
$T_{sop}$-locally trivial bibundles are precisely the Hilsum-Skandalis morphisms, see \cite[Def. 62]{metzler}. Many statements about Hilsum-Skandalis morphisms follow from the general theory of \cref{bibundles}, for example \cref{bicat-bibundles:c} is \cite[Prop. 64]{metzler}.

\item
For sheaves on $\Top$ one typically uses the traditional sheaf condition of \cref{traditional-sheaf} for the Grothendieck topology $T_{op}$. Since $T_{op}$ is superextensive and singletonizable, a presheaf is a  traditional sheaf if and only if it is a $T_{op}$-sheaf (\cref{comparison-sheaf}). A warning similar to \cref{traditional-sheaf-bad} applies.

\end{enumerate}

\section{Examples of  functors between sites}

\label{examples-of-functors}

\subsection{Inclusion of opens and submanifolds}

We consider the functors
\begin{equation}
\label{chain-cartesian-spaces}
\Open \to \Subman \to \Man
\end{equation}
which are  inclusion functors of full subcategories, and hence fully faithful. The fibre product in each category is the fibre product of smooth manifolds, and so  the functors \cref{chain-cartesian-spaces}  preserve them. The Grothendieck topologies have been obtained by restriction of $T_{\Man}$; hence, the functors \cref{chain-cartesian-spaces} are continuous by \cref{induced-topology-on-subcategory}. 

The functor $\Subman \to \Man$ has dense image and preserves and invert coproducts because it is an equivalence. It is also cocontinuous: as  $T_{\Subman}$ and $T_{\Man}$ are singleton, we can use \cref{check-cocontinuity}. If now $\pi:Y \to X \subset \R^{n}$ is a surjective submersion, we may choose an embedding $Y \incl \R^{l}$ and replace $Y$ by its image, getting a covering in $\Subman$.

The  functor $\Open \to \Subman$ is cocontinuous: suppose $\pi:Y \to U$ is a surjective submersion from an embedded submanifold $Y\subset \R^{m}$ to an open subset $U \subset \R^{k}$. We choose an open cover of $U$ by countable many open sets $U_n$ with sections $s_n: U_n \to Y$, with $n\in \N$. Then we consider the object $\tilde Y \subset \R^{k+1}$ in $\Open$ defined in \cref{dsfsdfsdf}; coming with smooth maps $s:\tilde Y \to Y$ and $\tilde\pi : \tilde Y \to U$ such that
$\pi \circ s=\tilde\pi$. 
It is clear that $\tilde\pi:\tilde Y \to U$ is locally split: on $U_n$ we have the map $x \mapsto (x,n)$. This shows cocontinuity.
Moreover, the inclusion functor $\Open \to \Subman$ has dense image because every smooth manifold $M$ of dimension $n$ can be covered by countable many opens sets $U_i \subset \R^{n}$, whose disjoint union $Y$ is a smooth manifold and hence can be embedded in some $\R^{m}$.
The corresponding map $Y \to M$ is a surjective local diffeomorphism and hence  $T_{susu}|_{\Subman}$-locally split. Finally, it inverts coproducts: if an open subset $U \subset \R^{n}$ is a disjoint union, then each connected component is also an open subset of $\R^{n}$.

We apply \cref{comparison-lemma} to obtain a standard result about sheaves on smooth manifolds:

\begin{proposition}
\label{sheaves-on-Open}
The functors \cref{chain-cartesian-spaces} induce equivalences of categories:
\begin{equation*}
\Sh(\Open,T_{op}) \cong \Sh(\Subman,T_{op}) \cong \Sh(\Man, T_{\Man})\text{.}
\end{equation*}
\end{proposition}

\subsection{The smooth diffeology functor}

The \emph{smooth diffeology functor}
\begin{equation*}
\mathscr{S}: \Man \to \Diff
\end{equation*}
is defined by setting $\mathcal{D}_{\mathscr{S}(M)}(U) := C^{\infty}(U,M)$, i.e., the plots with domain $U \in \Open$ are the smooth maps from $U$ to $M$. It has the following well-known properties:
\begin{enumerate}[(i)]

\item 
It is fully faithful.

\item
It preserves   products and coproducts.

\item
It preserves submanifolds: if $N \subset M$ is an embedded submanifold, then the subspace diffeology on $N \subset \mathscr{S}(M)$ coincides with $\mathscr{S}(N)$. 

\item
Since $\Diff$ is complete, $\mathscr{S}$ preserves universal morphisms.

\item
It extends to a fully  faithfully functor on Fr\'echet manifolds \cite{losik1}, and even to more general manifolds modelled on locally convex spaces \cite{Wockel2013}.  

\item
If $M$ is a smooth manifold, then the D-open sets of $\mathscr{S}(M)$ are precisely the open sets in the manifold topology; in other words, the composition $\mathscr{D}\circ \mathscr{S}: \Man \to \Top$, where $\mathscr{D}$ is the D-topology functor (see \cref{D-topology-functor}), is the manifold topology functor.  

\item
If $M$ is a smooth manifold, then $\mathcal{D}_{\mathscr{S}(M)} = \Yo_M|_{\Open}$; in other words, 
\begin{equation}
\label{fadsfaf}
\mathcal{D} \circ \mathscr{S} = \mathscr{I}^{*} \circ \Yo
\end{equation}
where $\mathcal{D}: \Diff \to \Sh(\Open, T_{op})$ is the embedding functor and $\mathscr{I}:\Open \to \Man$ is the inclusion.

\end{enumerate}

Concerning some of the classes of maps we have considered in $\Man$ and $\Diff$ we have the following lemma.

\begin{lemma}
\label{surjective-submersions-in-diff}
The following conditions for a smooth map $f:M \to N$ between smooth manifolds  are equivalent :
\begin{enumerate}

\item
$f$ is a surjective submersion

\item
$\mathscr{S}(f)$ is a surjective D-submersion

\item
$\mathscr{S}(f)$ is a submersion

\end{enumerate} 
\end{lemma}

\begin{proof}
The only non-trivial implication 3 $\to$ 1  is, e.g. proved in \cite{Schaaf2020a}. 
\end{proof}

\begin{proposition}
\label{properties-of-S}
The smooth diffeology functor $\mathscr{S}:\Man\to \Diff$ is continuous and cocontinuous for each of the Grothendieck topologies on $\Man$  considered in \cref{GTs-on-Man} and each of the Grothendieck topologies on $\Diff$ considered in \cref{comparison-GTs-on-Diff}. It has dense image if $\Diff$ is equipped with the Grothendieck topology $T_{sudu}$ of subductions. 
\end{proposition}

\begin{proof}
\cref{continuous-functor-coverings,surjective-submersions-in-diff} show that $\mathscr{S}$ is continuous w.r.t. to $T_{suDsu}$. By \cref{comparison-GTs-on-Diff}, $T_{suDsu} \prec T$ for all other Grothendieck topologies $T$. Then, continuity follows by \cref{identity-functor-continuous}. 

For cocontinuity, consider a subduction $\pi:Y \to \mathscr{S}(M)$. For $x\in M$, let $\varphi:U \to V$ be chart with $V \subset M$ on open neighborhood of $x$ and $U \in\Open$. Then,  $U \incl V \incl M$ is a plot of $\mathscr{S}(M)$, and since $\pi$ is a subduction, there exists an open subset $U_x \subset U$ with a lift $d_x:U_x \to Y$. We may refine the open cover $(U_x)_{x\in M}$  to a countable one. Then, \cref{check-cocontinuity} (part (b)) shows that $\mathscr{S}$ is cocontinuous for $T_{cop}$ and $T_{sudu}$. \cref{comparison-GTs-on-Diff,identitiy-functor-is-continuous} show that it is cocontinuous for all other Grothendieck topologies. 

In order to see the dense image, let $X$ be a diffeological space, we consider the plots $c:U \to X$ and  their \quot{nebula},
\begin{equation*}
N(X) := \coprod_{U\in\Open} \;\coprod_{c \in \mathcal{D}_X(U)} U\text{.}
\end{equation*} 
The corresponding map  $\pi: N(X) \to X$ is a subduction, but in general neither a submersion nor $D$-locally split. Hence, $\mathscr{S}$ has only a dense image when $\Diff$ is equipped with the Grothendieck topology of subductions.   
\end{proof}

\begin{corollary}
\label{smooth-diffeology-covariance}
The smooth diffeology functor $\mathscr{S}$ sends Lie groupoids to diffeological groupoids, and $T_{\Man}$-weak equivalences to $T$-weak equivalences, where $T$ is any of the Grothendieck topologies considered in \cref{comparison-GTs-on-Diff}. Likewise, it sends $T_{op}$-locally trivial principal $G$-bundles on $\Man$ to $T$-locally trivial principal $\mathscr{S}(G)$-bundles on $\Diff$, and the same for bibundles and anafunctors. 
\end{corollary}

\begin{proof}
This follows from the statement that $\mathscr{S}$ is continuous in combination with \cref{continuous-functors-and-weak-equivalences,functoriality-of-G-bundles,continuous-functors-and-anafunctors}. 
\end{proof}

\begin{corollary}
\label{smooth-diffeology-sheaves}
The smooth diffeology functor $\mathscr{S}$ induces an adjunction
\begin{equation*}
\mathscr{S}^{*}: \Sh(\Diff,T) \rightleftharpoons \Sh(\Man,T_{\Man}): \mathscr{S}_{*}
\end{equation*}
whose counit $\varepsilon : \mathscr{S}^{*} \circ \mathscr{S}_{*} \to \id$ is an isomorphism, for $T$ any  of the Grothendieck topologies on $\Diff$ considered in \cref{comparison-GTs-on-Diff}. This adjunction is an equivalence if $T=T_{sudu}$. 
\end{corollary}

We remark that the right Kan extension $\mathscr{S}_{*}$ of sheaves from manifolds to diffeological spaces has a particularly nice form.

\begin{lemma}
\label{extension-along-S}
For $\mathcal{F}$ a sheaf on $\Man$ and $X$ a diffeological space, there is a canonical bijection
\begin{align*}
(\mathscr{S}_{*}\mathcal{F})(X) 
\cong \mathrm{Hom}_{\PSh(\Open)}(\mathcal{D}_X,\mathcal{F}|_{\Open})
\end{align*}
that is
natural in $X$.
\end{lemma}

\begin{proof}
By \cref{Yoneda-extension}, the right Kan extension of $\mathcal{F}$ along the Yoneda embedding $\Yo^{\Sh}: \Man \to \Sh(\Man, T_{\Man})$ is 
\begin{equation*}
(\Yo^{\Sh}_{*}\mathcal{F})(\mathcal{G}) = \mathrm{Hom}_{\PSh(\Man)}(\mathcal{G},\mathcal{F})\text{.}
\end{equation*}
On the other hand, by \cref{fadsfaf}, the Yoneda embedding is canonically naturally isomorphic  the following composite:
\begin{equation*}
\alxydim{}{\Man \ar[r]^{\mathscr{S}} & \Diff \ar[r]^-{\mathcal{D}} & \Sh(\Open,T_{op}) \ar[r]^-{\sim} & \Sh(\Man, T_{\Man})\text{.}}
\end{equation*}
This shows that $(\mathscr{S}_{*}\mathcal{F})(X) = (\Yo_{*}\mathcal{F})|_{\Diff}\cong \mathrm{Hom}_{\PSh(\Open)}(\mathcal{D}_X,\mathcal{F}|_{\Open})$. 
\end{proof}

\begin{example}
Let $\Omega^k$ be the sheaf of differential $k$-forms on $\Man$. Then, computing $\mathscr{S}_{*}\Omega^k$ using \cref{extension-along-S} gives precisely the definition of differential forms on diffeological spaces.
Note that, by \cref{extension-preserves-sheaf-1,properties-of-S}, $\mathscr{S}_{*}\Omega^k$ is a sheaf for each of the Grothendieck topologies on $\Diff$ considered in \cref{comparison-GTs-on-Diff}.
\end{example}

\subsection{The D-topology functor}

\label{D-topology-functor}

The D-topology functor 
\begin{equation*}
\mathscr{D}: \Diff \to \Top
\end{equation*}
sends a diffeological space $X$ to the set $X$ equipped with the topology consisting of all D-open subsets. It is faithful, but not full.
The following are results of Christensen-Sinnamon-Wu \cite{Christensen}{}: 
\begin{enumerate}[(i)]

\item 
$\mathscr{D}$ has a right adjoint, and hence preserves  colimits.

\item
$\mathscr{D}$ does not preserve subspaces: if $A \subset X$ is a subset, then the subspace topology $A \subset D(X)$ may be finer than the D-topology of $A$. 

\item
$\mathscr{D}$ does not preserve all products.

\item
$\mathscr{D}$ does not preserve mapping spaces; the D-topology of $C^{\infty}(X,Y)$ is between the weak and strong topologies. 
\end{enumerate}
Since $D$ does not preserve products, it follows that diffeological categories / groupoids / groups have in general no underlying topological counterparts. A possible solution was found by the work of Christensen-Sinnamon-Wu \cite{Christensen}{}, Kihara \cite{Kihara2018}{}, Shimakawa-Yoshida-Haraguchi \cite{Shimakawa2010}{}, featured by the observation that the D-topology of every diffeological space $X$ is $\Delta$-generated, i.e., it is final w.r.t. all continuous maps from simplices $\Delta^k$ to $X$. The co-restriction
\begin{equation*}
\mathscr{D}^{\Delta}: \Diff \to \Top^{\Delta}
\end{equation*}
to $\Delta$-generated topological spaces is in fact surjective,  preserves all colimits, but now also preserves products. This guarantees, at least, that diffeological groups are sent to topological groups.

Another relevant observation is the following lemma
proved in \cite[\S 2.18]{iglesias1}:

\begin{lemma}
If $f: X \to Y$ is a submersion between diffeological spaces, then $\mathscr{D}(f)$ is an open map. 
\end{lemma}

This provides the first part of \cref{continuous-functor-coverings}, heading to a possible proof that $\mathscr{D}^{\Delta}$ is continuous. The second part (it preserves universal morphisms) is also true, as both categories are complete. Just the third part remains unclear (to the author): does $\mathscr{D}^{\Delta}$ preserve all fibre products?

\subsection{The Yoneda embedding}

The \emph{Yoneda embedding} \begin{equation*}
\Yo: \mathscr{C} \to \PSh(\mathscr{C}):X \mapsto \Yo_X
\end{equation*}
is  fully faithful  and preserves all limits. The following result is similar to  \cite[C2.2.7]{Johnstone2002}.

\begin{proposition}
\label{yoneda-continuous}
Let $(\mathscr{C},T)$ be a  site,
and let $\PSh(\mathscr{C},T)$ be equipped with the induced Grothendieck topology $\Pre(T)$. Then, the Yoneda embedding $\Yo$ is continuous. If $T$ is singleton, it is cocontinuous. 
\end{proposition}

\begin{proof}
For continuity, since $\PSh(\mathscr{C})$ is complete, $\Yo$ preserves limits, and $\Pre(T)$ is singleton, it suffices to show that $\Yo$ sends $T$-locally split morphisms to $\Pre(T)$-coverings. Let $\pi:Y \to X$ be $T$-locally split, let $X'\in \mathscr{C}$, and let $\psi \in \Yo_X(X')$, i.e.,  $\psi:X' \to X$. We choose  a $T$-covering $(\pi_i: U_i \to X)_{i\in I}$ with local sections $\rho_i: U_i \to Y$. For each $i\in I$, we consider the diagram
\begin{equation*}
\alxydim{}{U'_i \ar[r]^{p_i} \ar[d]_{\pi'_i} & U_i \ar[r]^{\rho_i} \ar[d]_{\pi_i} & Y \ar[dl]^{\pi} \\ X' \ar[r]_{\psi} & X}
\end{equation*}
whose left part is a pullback. 
The family $(\pi_i':U'_i \to X')_{i\in I}$ is a $T$-covering of $X'$, and $\psi_i:=\rho_i \circ p_i \in \Yo_{Y}(U_i')$ are elements such that $\Yo_{\pi}(\psi_i')=\pi \circ \rho_i\circ p_i=\psi \circ \pi_i' = \pi_i'^{*}\psi$. This is the condition showing that $\Yo_{\pi}$ is a $\Pre(T)$-covering.

Next we show cocontinuity. Since $T$ is now singleton, we can apply \cref{check-cocontinuity}, Part (a). Let $X$ be an object and $\phi: \mathcal{F} \to \Yo_X$ be a $\Pre(T)$-covering. 
We apply the condition for $\Pre(T)$-coverings to the object $X$ and $\psi=\id_X\in \Yo_X(X)$. Thus, there exists a $T$-covering $\pi:Y \to X$ and an element $\psi'\in \mathcal{F}(Y)$ such that $\phi(\psi')=\Yo_{\pi}(\psi)$. This means that $\psi'$ corresponds to a presheaf morphism $f: \Yo_Y \to \mathcal{F}$ such that $\phi \circ f=\Yo_{\pi}$.   
\end{proof}

If $T$ is a subcanonical Grothendieck topology on $\mathscr{C}$, then by \cref{subcanonical-and-sheaves} the Yoneda embedding co-restricts to a  functor
\begin{equation*}
\Yo^{\Sh}: \mathscr{C} \to \Sh(\mathscr{C},T)\text{.}
\end{equation*}
This functor is still fully faithful and preserves limits. We recall from \cref{presheaves-and-sheaves} that $\Pre(T)|_{\Sh(\mathscr{C},T)}$ is the canonical Grothendieck topology on $\Sh(\mathscr{C},T)$. Thus, we get the following consequence of \cref{yoneda-continuous}:

\begin{corollary}
Let $(\mathscr{C},T)$ be a  subcanonical site,
and let $\Sh(\mathscr{C},T)$ be equipped with the canonical Grothendieck topology. Then, the Yoneda embedding $\Yo^{\Sh}$ is continuous, and cocontinuous if $T$ is singleton. 
\end{corollary}

Currently, the author does not know whether or not $\Yo$ or $\Yo^{\Sh}$ preserve or invert coproducts; thus, we cannot apply \cref{extension-preserves-sheaf-1}.
However, one can easily compute the right Kan extension of a sheaf $\mathcal{F}\in \Sh(\mathscr{C},T)$.

\begin{lemma}
\label{Yoneda-extension}
Let $\mathcal{F}$ be a sheaf on a site $(\mathscr{C},T)$. We have
\begin{equation*}
\Yo^{\Sh}_{*}\mathcal{F} = \Yo_{\mathcal{F}}\text{.}
\end{equation*}
In particular, $\Yo^{\Sh}_{*}\mathcal{F}$ is a sheaf. 
\end{lemma}

\begin{proof}
The set-theoretic realization of the limit \cref{right-Kan-extension} directly gives 
\begin{equation*}
(\Yo^{\Sh}_{*}\mathcal{F})(\mathcal{G}) = \mathrm{Hom}_{\PSh(\mathscr{C})} (\mathcal{G},\mathcal{F})= \mathrm{Hom}_{\Sh(\mathscr{C})} (\mathcal{G},\mathcal{F})=\Yo_{\mathcal{F}}(\mathcal{G})\text{.}
\end{equation*} 
Since the result is a representable sheaf, and we equipped $\Sh(\mathscr{C},T)$ with the canonical Grothendieck topology $T_{\Sh(\mathscr{C},T)}$, we see from \cref{subcanonical-and-sheaves} that $\Yo^{\Sh}_{*}\mathcal{F}$ is again a sheaf. 
\end{proof}

\bibliographystyle{kobib}
\bibliography{kobib}

\end{document}

%% file: styles.tex
\usepackage{amsthm}
\usepackage{thmtools}
\usepackage{graphicx}
\usepackage[margin=2cm,font=normalsize,labelfont=bf]{caption}
\usepackage{verbatim}
\usepackage[shortlabels]{enumitem}
\usepackage{fancyhdr}
\usepackage[]{geometry}
\usepackage{xstring}
\usepackage{needspace}
\usepackage{color}
\usepackage{amstext}
\usepackage{nameref}
\usepackage[hypertexnames=false]{hyperref}
\usepackage{mathdots}
\usepackage{soul}
\usepackage{chngcntr}
\usepackage[section]{placeins}

\makeatletter
\newcounter{denseversion}
\newcounter{comments}
\newcounter{authorcounter}
\newcounter{adresscounter}

\def\title#1{\gdef\@title{#1}}
\def\@title{}
\def\subtitle#1{\gdef\@subtitle{#1}}
\def\@subtitle{}

\def\authortagsused{0}
\def\adresstag#1{\if!#1!\else$^{\;#1\;}$\fi}
\def\@authorsep#1{
  \ifnum\value{authorcounter}=#1 and \else\unskip, \fi
}

\renewcommand{\author}[2][]{
  \stepcounter{authorcounter}
  \if!#1!\else\gdef\authortagsused{1}\fi
  \ifnum\value{authorcounter}=1
    \def\@authorstringa{#2\adresstag{#1}}
    \def\@authorstringb{#2}
    \def\@authorstringc{#2\adresstag{#1}}
  \else
    \ifnum\value{authorcounter}=2
      \g@addto@macro\@authorstringa{\@authorsep{2}#2\adresstag{#1}}
      \g@addto@macro\@authorstringb{\@authorsep{2}#2}
      \g@addto@macro\@authorstringc{\\#2\adresstag{#1}}
    \else
      \g@addto@macro\@authorstringa{\@authorsep{3}#2\adresstag{#1}}
      \g@addto@macro\@authorstringb{\@authorsep{3}#2}
      \g@addto@macro\@authorstringc{\\#2\adresstag{#1}}
    \fi
  \fi}
\def\@author{\ifnum\value{denseversion}=0\@authorstringa\else\@authorstringb\fi}

\def\@adressstringa{}
\def\@adressstringb{}
\newcommand{\adress}[2][]{
  \stepcounter{adresscounter}
  \ifnum\value{adresscounter}=1
    \g@addto@macro\@adressstringa{\ifnum\authortagsused=0\def\br{\\}\else\def\br{, }\fi\adresstag{#1}#2}
    \g@addto@macro\@adressstringb{\def\br{\\}\adresstag{#1}\parbox[t]{14cm}{#2}}
  \else
    \g@addto@macro\@adressstringa{\\[\bigskipamount]\adresstag{#1}#2}
    \g@addto@macro\@adressstringb{\\[\medskipamount]\adresstag{#1}\parbox[t]{14cm}{#2}}
  \fi}

\def\preprint#1{\gdef\@preprint{#1}}
\def\@preprint{}
\def\keywords#1{\gdef\@keywords{#1}}
\def\@keywords{}
\def\msc#1{\gdef\@msc{#1}}
\def\@msc{}
\def\email#1{
   \gdef\@email{#1}
   \g@addto@macro\@authorstringc{ {\it (#1)}}}
\def\@email{}
\def\dedication#1{\gdef\@dedication{#1}}
\def\@dedication{}

\def\mybaselinestretch#1{
  \gdef\@mybaselinestretch{#1}
  \renewcommand{\baselinestretch}{\@mybaselinestretch}}

\def\myparskip#1{
  \gdef\@myparskip{#1}
  \setlength{\parskip}{\@myparskip}}

\mybaselinestretch{1.2}
\myparskip{1.5ex}

\binoppenalty=\maxdimen
\relpenalty=\maxdimen

\normalfont
\normalsize
\newlength{\@listleftmargin}
\def\setenumstandard{
  \setlist{leftmargin=\@listleftmargin,itemsep=0pt,topsep=0pt,partopsep=0pt,parsep=\@myparskip}
  \setlist[enumerate]{align=left,labelsep=*,leftmargin=\@listleftmargin,itemsep=0pt,topsep=0pt,partopsep=0pt,parsep=\@myparskip}
}

\def\denseversion{
  \setcounter{denseversion}{1}
  \newgeometry{left=3cm,right=3cm,top=3cm}
  \mybaselinestretch{1.1}
  \myparskip{0.8ex}
  \normalfont
  \def\possiblelinebreak{}
  \fancyfoot[C]{\itshape{--$\,\,$\thepage$\,\,$--}}}

\def\possiblelinebreak{\\}

\renewcommand{\emph}[1]{\def\reserved@a{it}\ifx\f@shape\reserved@a\ul{#1}\else\textit{#1}\fi}

\def\setcrefnames{}
\newcommand{\mytableofcontents}{
   \ifnum\value{denseversion}=0
     \tableofcontents
     \setcrefnames 
   \else
     \renewcommand{\baselinestretch}{1.1}
     \setlength{\parskip}{0ex}
     \normalfont
     \begingroup
     \def\addvspace##1{\vskip0.4em}
     \tableofcontents
     \setcrefnames 
     \endgroup
     \renewcommand{\baselinestretch}{\@mybaselinestretch}
     \setlength{\parskip}{\@myparskip}
     \normalfont
   \fi}

\newlength{\zeilenlaenge}
\def\putindent#1{
  \settowidth{\zeilenlaenge}{#1}
  \ifnum\zeilenlaenge>\textwidth
    #1
  \else
    \noindent #1
  \fi
}

\hypersetup{
    linktocpage = true,
    colorlinks,
    linkcolor=[RGB]{0,0,120},
    citecolor=[RGB]{0,0,120},
    urlcolor=[RGB]{0,0,120}
}

\def\pdfdaten{
  \hypersetup{
    pdftitle = {\@title},
    pdfauthor = {\@author},
    pdfkeywords = {\@keywords},    
    bookmarksopen = true,
    bookmarksopenlevel = 1
  }}  

\def\showkeywords{\begin{flushleft}\footnotesize\textbf{Keywords}: \@keywords\end{flushleft}}
\def\showmsc{\begin{flushleft}\footnotesize\textbf{MSC 2010}: \@msc\end{flushleft}}

\def\tocsection#1{\section*{#1}\addcontentsline{toc}{section}{#1}}

\def\mytitle{}
\def\zmptitle{
  \begin{tabular}{cc}
    \begin{minipage}[c]{0.4\textwidth}
      \begin{flushleft}
        \includegraphics[width=110pt]{../../tex/zmp}
      \end{flushleft}  
    \end{minipage}&
    \begin{minipage}[c]{0.55\textwidth}
      \begin{flushright}
      {\small\sf\@preprint}
      \end{flushright}
    \end{minipage}
  \end{tabular}
  \vskip 2cm}

\def\maketitle{
  \pdfdaten
  \noindent
  \mytitle
  \begin{center}
    \LARGE\@title\\
    \if!\@subtitle!\else\smallskip\LARGE\@subtitle\\\fi
    \bigskip
    \if!\@author!\else\bigskip\large\@author\\\fi
    \ifnum\value{denseversion}=0
      \if!\@adressstringa!\else\bigskip\normalsize\@adressstringa\\\fi
      \if!\@email!\else\ifnum\value{authorcounter}=1\bigskip\normalsize\textit{\@email}\\\else\fi\fi
    \else
    \fi
    \if!\@dedication!\else\bigskip\normalsize{\@dedication}\\\fi
  \end{center}
  \ifnum\value{denseversion}=0\vskip 1.5cm\else\vskip0.5cm\fi}

\def\kobib#1{
  \begin{raggedright}
  \ifnum\value{denseversion}=0\else\small\fi
  \Oldbibliography{#1/kobib}
  \bibliographystyle{#1/kobib}
  \end{raggedright}
  \ifnum\value{denseversion}=0\else
      \noindent
      \if!\@authorstringc!\else
        \ifnum\authortagsused=0\ifnum\value{authorcounter}>1\normalsize\@authorstringc\\[\medskipamount]\else\fi\else\normalsize\@authorstringc\\[\medskipamount]\fi
      \fi
      \if!\@adressstringb!\else\normalsize\@adressstringb\\{}\fi
      \ifnum\authortagsused=0
        \ifnum\value{authorcounter}=1
          \if!\@email!\else\linebreak\normalsize\textit{\@email}\\{}\fi
        \else
        \fi
      \else
      \fi
  \fi}

\let\Oldbibliography\bibliography
\def\bibliography#1{
  \begin{raggedright}
  \ifnum\value{denseversion}=0\else\small\fi
  \Oldbibliography{#1}
  \end{raggedright}
  \ifnum\value{denseversion}=0\else
      \medskip
      \noindent
      \if!\@authorstringc!\else
        \ifnum\authortagsused=0\ifnum\value{authorcounter}>1\normalsize\@authorstringc\\[\medskipamount]\else\fi\else\normalsize\@authorstringc\\[\medskipamount]\fi
      \fi
      \if!\@adressstringb!\else\normalsize\@adressstringb\\{}\fi
      \ifnum\authortagsused=0
        \ifnum\value{authorcounter}=1
          \if!\@email!\else\linebreak\normalsize\textit{\@email}\\{}\fi
        \else
        \fi
      \else
      \fi
  \fi
}

\newenvironment{commentfigure}{}
\newenvironment{sidewayscommentfigure}{\begin{minipage}}{\end{minipage}}
\newenvironment{displaycomment}{\begin{list}{}{\rightmargin=1cm\leftmargin=1cm}\item\sf\begin{small}}{\end{small}\end{list}}

\def\tocmark#1{}

\def\draftstamp#1{
  \def\tocmark##1{
    \ifnum\c@secnumdepth=0\section{##1}\fi
    \ifnum\c@secnumdepth=1\subsection{##1}\fi
    \ifnum\c@secnumdepth=2\subsubsection{##1}\fi
    \ifnum\c@secnumdepth=3\subsubsection{##1}\fi
  }
  \ifnum\value{comments}=0
    \gdef\@draft{DRAFT - Edited on \today\ by #1 - Comments are not displayed}
  \else
    \gdef\@draft{DRAFT - Edited on \today\ by #1 - Comments are displayed}
  \fi
  \fancyhead[C]{\footnotesize\tt\textcolor{red}{\@draft}}}

\def\skript{
  \renewenvironment{displaycomment}{}{}
  \ifnum\value{comments}=0
    \renewenvironment{example*}{\comment}{\endcomment}
    \renewenvironment{remark*}{\comment}{\endcomment}
  \else\fi
  \parindent=0mm	
}

\newcounter{sectioning}
\def\tsection#1{\ifnum\value{sectioning}=0\section{#1}\fi}
\def\lsection#1{
  \ifnum\value{sectioning}=1
    \clearpage
    \def\thesection{\lektionname~\arabic{section}:}
    \section{#1}
    \def\thesection{\arabic{section}}
  \fi}
\def\tsubsection#1{\ifnum\value{sectioning}=0\subsection{#1}\fi}
\def\lsubsection#1{\ifnum\value{sectioning}=1\subsection{#1}\fi}
\def\tsubsubsection#1{\ifnum\value{sectioning}=0\subsubsection{#1}\fi}
\def\lsubsubsection#1{\ifnum\value{sectioning}=1\subsubsection{#1}\fi}

\mybaselinestretch{}
\setenumstandard

\makeatother

%% file: makros.tex
\usepackage{amssymb}
\usepackage{amsmath}
\usepackage{amstext}
\usepackage{mathrsfs}
\usepackage{euscript}
\usepackage{color}
\usepackage[all,cmtip,color,arrow]{xy}
\usepackage{xstring}
\usepackage{slashed}
\usepackage{tensor}
\usepackage{tikz}

\entrymodifiers={+!!<0pt,\fontdimen22\textfont2>}

\def\N {\mathbb{N}}
\def\Z {\mathbb{Z}}

\def\R {\mathbb{R}}

\def\T {\mathbb{T}}

\def\id{\mathrm{id}}

\def\subset{\subseteq}

\def\sep{\;|\;}

\renewcommand{\varepsilon}{\epsilon}

\newcommand{\incl}{\hookrightarrow}

\def\gdw{\Leftrightarrow}

\makeatletter
\renewenvironment{proof}[1][\nameProof]
  {\par\pushQED{\qed}%
   \normalfont \topsep6\p@\@plus6\p@\relax
   \trivlist
   \item[\hskip\labelsep
         \itshape
         #1\@addpunct{.}]
  \leavevmode}
  {\popQED\endtrivlist\@endpefalse}
\makeatother

\def\notebox#1#2{\begin{minipage}[b]{#1}\sloppy\renewcommand{\baselinestretch}{0.8}\footnotesize \begin{center}#2\end{center}\end{minipage}}

\def\mqquad{\hspace{-4em}}

\newcommand{\arr}[1][r]{\ar@<0.7ex>[#1]\ar@<-0.7ex>[#1]}
\newcommand{\arrr}[1][r]{\ar@<1.4ex>[#1]\ar[#1]\ar@<-1.4ex>[#1]}
\newcommand{\arrrr}[1][r]{\ar@<2.1ex>[#1]\ar@<-2.1ex>[#1]\ar@<0.7ex>[#1]\ar@<-0.7ex>[#1]}

\def\stackref#1#2{\stackrel{\text{\ref{#1}}}{#2}}
\def\eqref#1{\stackref{#1}{=}}

\newlength{\myeqt} 
\newlength{\myeqs} 
\newlength{\myeqm} 
\newlength{\myeqn} 
\newcommand\symtext[3][\myeqn]{
  \settowidth{\myeqt}{#2}
  \settowidth{\myeqs}{$#3$}
  \addtolength{\myeqs}{\the\myeqm}
  \ifdim\myeqt>\myeqs
    \stackrel{\hspace{-#1}\notebox{#1}{\medskip #2 \\ $\downarrow$\smallskip}\hspace{-#1}}{#3}
  \else
    \stackrel{\text{#2}}{#3}
  \fi}

\def\brackets#1{\IfStrEq{#1}{-}{}{(#1)}}
\def\subindex#1{\IfStrEq{#1}{-}{}{_{#1}}}

\newcommand{\alxydim}[2]{\begin{aligned}\xymatrix#1{#2}\end{aligned}}



\newlength{\myl}

\def\ddt#1#2#3{\left.\frac{\mathrm{d}^{\IfStrEq{#1}{1}{}{#1}}}{\mathrm{d}#2}\IfStrEq{#2}{#3}{\right.}{\right|_{#3}}}

\def\Man{\mathscr{M}\mathrm{an}}

\def\Diff{\mathscr{D}\mathrm{iff}}

\def\Top{\mathscr{T}\mathrm{op}}

\def\PSh{\mathscr{PS}\mathrm{h}}
\def\Sh{\mathscr{S}\mathrm{h}}

\def\Open{\mathscr{O}\mathrm{pen}}

\def\Set{\mathscr{S}\mathrm{et}}

\def\grpd{\mathcal{G}\!rpd}

\def\des{\mathcal{D}\!esc}

\def\Hom{\mathscr{H}\mathrm{om}}

\def\pr{{\mathrm{pr}}}

\newlength{\widthtmp}
\def\length#1{\settowidth{\widthtmp}{#1}\the\widthtmp}

\def\ttimes#1#2{\hspace{-0.15em}\tensor[_{#1}]{\times}{_{#2}}}



%% file: hyphenation_en.tex
\usepackage[latin1]{inputenc}
\usepackage[english]{babel}
\usepackage[english]{cleveref}

\def\refname{References}

\def\quot#1{``#1''}

\def\quand{\quad\text{ and }\quad}
\def\quomma{\quad\text{, }\quad}

\def\nameTheorem{Theorem}
\def\nameTheorems{Theorems}
\def\nameDefinition{Definition}
\def\nameDefinitions{Definitions}
\def\nameCorollary{Corollary}
\def\nameCorollaries{Corollaries}
\def\nameProposition{Proposition}
\def\namePropositions{Propositions}
\def\nameConstruction{Construction}
\def\nameConstructions{Constructions}
\def\nameLemma{Lemma}
\def\nameLemmas{Lemmas}
\def\nameRemark{Remark}
\def\nameRemarks{Remarks}
\def\nameProblem{Problem}
\def\nameProblems{Problems}
\def\nameExample{Example}
\def\nameExamples{Examples}
\def\nameExercise{Exercise}
\def\nameExercises{Exercises}
\def\nameWarnung{Warning}
\def\nameWarnungs{Warnings}

\def\nameProof{Proof}
\def\nameEquation{Eq.}
\def\nameEquations{Eqs.}
\def\nameSection{Section}
\def\nameSections{Sections}
\def\nameAppendix{Appendix}
\def\nameAppendixs{Appendices}
\def\nameFigure{Figure}
\def\nameFigures{Figures}